\documentclass{article}

\usepackage[T1]{fontenc}
\usepackage[utf8]{inputenc}
\ifdefined\DeclareUnicodeCharacter
  \DeclareUnicodeCharacter{2010}{-}
\fi
\usepackage{lmodern}
\usepackage{libertine}
\usepackage{xcolor}
\definecolor{linkblue}{RGB}{0,90,180}
\definecolor{lightlinkblue}{RGB}{0,90,180}
\usepackage[a4paper,left=3cm,right=3cm,top=3cm,bottom=3cm]{geometry}

\usepackage{booktabs}
\usepackage{float}
\usepackage{bm}
\usepackage{makecell}
\usepackage{bbm}
\usepackage{enumitem}
\usepackage{mathrsfs}
\usepackage{amsmath,amssymb,amsfonts}
\usepackage{upgreek}
\usepackage{mathtools}
\usepackage{multirow,multicol}
\usepackage{amsthm}
\usepackage{comment}
\usepackage{graphicx}
\usepackage[export]{adjustbox}
\usepackage{epstopdf}
\usepackage{pgfplots}
\usepackage{algorithm}
\usepackage{algorithmic}
\usepackage[font=small,labelfont=bf]{caption}
\usepackage{subcaption}
\usepackage{quoting}
\usepackage{xspace}
\usepackage{pifont}
\usepackage{lineno}
\usepackage[numbers,sort&compress]{natbib}

\usepackage[colorlinks=true,bookmarksopen=true,
    linkcolor=lightlinkblue,citecolor=linkblue,urlcolor=linkblue]{hyperref}
\usepackage{cleveref}
\theoremstyle{plain}
\newtheorem{theorem}{Theorem}[section]

\newtheorem{lemma}[theorem]{Lemma}
\newtheorem{corollary}[theorem]{Corollary}

\theoremstyle{definition}

\newtheorem{assumption}{Assumption}

\theoremstyle{remark}
\newtheorem{remark}[theorem]{Remark}

\crefname{hypothesis}{Hypothesis}{Hypotheses}
\crefname{fact}{Fact}{Facts}

\DeclareMathOperator{\Var}{Var}
\DeclareMathOperator{\supp}{supp}
\def\R{\mathbb R}
\def\N{\mathbb N}

\def\to{\rightarrow}

\title{Random domain decomposition for parabolic PDEs on graphs\thanks{M. Hern\'{a}ndez has been funded by the Transregio 154 Project, Mathematical Modelling, Simulation, and Optimization Using the Example of Gas Networks of the DFG, project C07, the fellowship ``ANID-DAAD bilateral agreement'', and projects DAAD/CAPES grant 57703041 ``Control and numerical analysis of complex systems'', and DFG/NRF S\"udkorea-NRF-DFG-2023 programme, grant 530756074, and NSF grants No 2529107.}}
\author{Mart\'{i}n Hern\'{a}ndez\thanks{Department of Statistics and Data Science, University of California, Los Angeles, CA 90095 Los Angeles, USA. Email: \texttt{martinh@ucla.edu}.}}
\date{\vspace{-1cm}}
\hypersetup{
  pdftitle={Random domain decomposition for parabolic PDEs on graphs},
  pdfauthor={Martin Hernandez},
  pdfkeywords={Partial Differential Equations on Graphs, Random Batch Method, Domain Decomposition}
}

\begin{document}
\maketitle
\begin{abstract}
The simulation of complex systems, such as gas transport in large pipeline networks, often involves solving PDEs posed on intricate graph structures. Such problems require considerable computational and memory resources. The \emph{Random Batch Method} (RBM) has shown promise in addressing these challenges via stochastic decomposition techniques. In this paper, we apply the RBM at the PDE level for parabolic equations on graphs, {without assuming any preliminary discretization in space or time}. We consider a non-overlapping domain decomposition in which the PDE coefficients and source terms are randomized. We prove that the resulting RBM-based scheme converges, in the mean-square sense and uniformly in time, to the true PDE solution with first-order accuracy in the RBM step size. Numerical experiments confirm this convergence rate and demonstrate reductions in memory usage and computational time per RBM realization compared with solving on the full graph. These advantages persist across different time discretization schemes. Finally, we discuss the sampling variability of the method and the effective computational time associated with empirical averaging of independent RBM realizations.
\end{abstract}

\noindent\textbf{Keywords.} Partial Differential Equations on Graphs, Random Batch Method, Domain Decomposition.

\noindent\textbf{MSC codes.} 35R02, 65C99, 65M55, 65M75, 68Q25

\section{Introduction}

\subsection{Motivation}
Modeling and simulating physical phenomena such as gas transport or heat transfer in large, complex networks often leads to partial differential equations (PDEs) on graph structures \cite{MR3728372,Leuge,Mugnolo,wad1}. PDEs on graphs have been extensively studied for diffusion and transport over networked domains \cite{Leuge,Mugnolo,wad1,wad3,MR2818413}, and in singular‐limit regimes \cite{Egger,barcena2025control} and for control of parabolic and hyperbolic equations \cite{MR4421539,MR4581791,dager2006wave}. Nonetheless, the high complexity induced by the network size and topology can make classical solvers prohibitively expensive in both time and memory \cite{MR4175147}. A common strategy to reduce memory usage is domain decomposition \cite{MR2093789}. Within the context of graphs, both overlapping (Schwarz‐type) and non‐overlapping (Schur‐complement) methods have been developed to reduce the solver cost; see \cite[Chapter III]{MR2093789} and \cite{MR1857663} for overviews.  However, these methods rely on an outer iteration in which one alternates between solving the PDE on each subdomain (subgraph) with Dirichlet (or Neumann) data borrowed from its neighbors and then updating those interface data before the next pass. Although each local solve is smaller, repeated synchronization of boundary data across all subdomains until global convergence can dominate runtime, especially as the number of subdomains grows.

On the other hand, \emph{Random Batch Method} (RBM)—inspired by minibatch gradient descent—was introduced in \cite{ShietalRBM,MR4230431} to reduce the cost of simulating large‐scale interacting particle systems by randomly grouping particles into small batches over short temporal intervals. RBM shares with domain decomposition the idea of splitting the problem into subproblems, but rather than solving every subproblem at each step, RBM randomly selects a single subproblem in each time window. Beyond its original formulation, RBM has also been used in a broad range of large-scale applications, including molecular dynamics, statistical sampling, agent-based models for collective behavior, and quantum Monte Carlo methods; see the overview \cite{JinLi2021RBMReview}.
Subsequent extensions of RBM encompass finite‐dimensional dynamical systems \cite{MR4433122,MR4307003,veldman2023stability}, time‐discrete PDEs and gradient flows \cite{eisenmann2022randomized,Corella_hernandez2025}, and spatially discrete PDEs \cite{firstpaper}. However, each of these variants remains tied to a particular mesh or time‐discretization scheme, which restricts their applicability. In addition, the hallmark savings in memory and CPU time should emerge from the RBM principle itself, rather than from particular discretization choices. To address these issues, we introduce a continuous RBM at the PDE level, decoupling the RBM from any discretization to enable mesh‐independent convergence proofs and intrinsic computational savings.

\subsection{Contribution} 
Consider a connected finite metric graph $\mathcal{G} := (\mathcal{E}, \mathcal{V})$ with vertices $v \in \mathcal{V}$ and edges $e \in \mathcal{E}$. Denote by $\mathcal{V}_b$ and $\mathcal{V}_0$ the boundary and interior vertices of $\mathcal{G}$, and $\mathcal{E}(v)$ the edges adjacent to the vertex $v$. Let $T > 0$, and over each edge $e \in \mathcal{E}$, we consider the parabolic equation
\begin{align}\label{eq:heat1}
\begin{cases}
\partial_t y_e - \partial_{x}(a_e\partial_x y_e)+ b_e\partial_x y_e +p_e y_e = f_e, & (x,t) \in e \times (0, T), \\
y_{e'}(v, t) = y_{e''}(v, t), & (v,t) \in \mathcal{V}_0\times(0,T),\,\forall e', e'' \in \mathcal{E}(v), \\
\sum_{e \in \mathcal{E}(v)}a_e(v) \partial_x y_e(v, t) n_e(v) = 0, & 
(v,t) \in \mathcal{V}_0\times(0,T),\\
y_e(v, t) = g_v(t), & (v,t) \in \mathcal{V}_b\times(0,T),\, e \in \mathcal{E}(v), \\
y_e(x, 0)= y_e^0(x), & x \in e,
\end{cases}
\end{align}
where $y_e^0$ is the initial condition, $f_e$ the source term, $g_v$ the boundary data, and $a_e$, $b_e$, and $p_e$ are the diffusion, convection, and potential coefficients, respectively. Here $n_e: \mathcal{V} \to \{-1, 0, 1\}$ corresponds to the incidence vector, defined as
\begin{align*}
    n_e(v)=-1, \quad \text {if } e=(v, \cdot), \quad n_e(v)=1, \quad \text {if } e=(\cdot, v), \quad \text {and} \quad n_e(v)=0 \quad \text {otherwise. }
\end{align*}
Conditions $\eqref{eq:heat1}_2$–$\eqref{eq:heat1}_3$ correspond to continuity and the flux (Kirchhoff) condition at the interior vertices $\mathcal{V}_0$, while $\eqref{eq:heat1}_4$–$\eqref{eq:heat1}_5$ correspond to the boundary condition at $\mathcal{V}_b$ and the initial condition in $\mathcal{E}$.

Our main contributions are:
\begin{enumerate}
  \item \textbf{Continuous RBM on graphs.} We introduce an RBM for PDEs defined on graphs, specifically for \eqref{eq:heat1}. Our method relies on the following steps: First, we introduce a non-overlapping decomposition of the graph into small subgraphs. Then, on each temporal window of length $h$, we randomly select one subgraph to evolve a localized version of \eqref{eq:heat1} while keeping the solution constant on the rest of the graph (freezing). This freeze-and-evolve mechanism yields a randomized domain decomposition of the full-graph dynamics.

  \item \textbf{Well-posedness and maximal regularity.}
  We establish existence, uniqueness, and maximal regularity for the deterministic advection--diffusion problem \eqref{eq:heat1} on a metric graph with Kirchhoff transmission conditions. The result applies to a non-self-adjoint operator with variable diffusion, convection, and potential coefficients. This well-posedness theory is then used to prove that the continuous RBM system is also well-posed.

  \item \textbf{Convergence estimates.}
  We prove that the mean-square error between the RBM solution and the true solution decays like $O(h)$, and we derive probability bounds ensuring each realization of the random system stays close to the exact solution of \eqref{eq:heat1} for sufficiently small $h$.

  \item \textbf{Numerical examples.}
  We present numerical simulations to demonstrate the $O(h)$ convergence rate. The results show that the Continuous RBM achieves accuracy comparable to solving on the full graph. Furthermore, we observe reductions in CPU time and memory usage per RBM realization compared with full-graph solvers. These results are illustrated using the implicit Euler scheme, the $\theta$-method, the Crank--Nicolson scheme, and the semi-implicit method. We also discuss the sampling variability of independent RBM realizations and the effective computational time associated with empirical averaging.
\end{enumerate}
\subsection{Related work}\label{sec:rel_literature}
RBM was first proposed in \cite{ShietalRBM,MR4230431} for large‐scale interacting particle systems, grouping particles into random batches to reduce complexity. In \cite{MR4433122}, it was extended to finite‐dimensional linear systems—via decomposition of the system matrix—and to the linear‐quadratic regulator. Subsequent work \cite{MR4307003,veldman2023stability} applies continuous RBM to model predictive control in finite dimensions.

Numerical methods for quantum graphs have been studied in \cite{ArioliBenzi2018}, combining finite element discretizations with Schur-complement reductions and Krylov-based exponential integrators. Our approach is complementary, since the random decomposition is introduced at the continuous PDE level rather than at the level of the discretized full-graph system. For PDEs, existing RBM schemes typically rely on a preliminary time discretization. In \cite{eisenmann2022randomized}, an RBM framework for nonlinear PDEs is introduced, viewed as an overlapping random domain decomposition. However, the authors first discretize the PDE in time using an implicit Euler scheme and then apply RBM, which forces the RBM parameter $h$ to coincide with the time discretization parameter $\Delta t$. They prove that the expected error between the randomized and exact solutions is $O(h)=O(\Delta t)$. Using similar ideas, \cite{MR4828491} introduces a time-discretized scheme for stochastic evolution equations and establishes convergence. In \cite{Corella_hernandez2025}, the same $O(h)$ rate is proved for finite-dimensional gradient flows and extended to infinite-dimensional flows, again, after temporal discretization via the minimizing movement scheme. In \cite{firstpaper}, a “discretize-RBM” approach is proposed for PDEs on graphs: this scheme, based on spatial discretization, yields expected‐error convergence under requirements on both $h$ and the spatial mesh size $\Delta x$. RBM has also been extended to the mean‐field limit of particle systems, notably to approximate kinetic equations such as the Landau and Fokker–Planck equations defined on the whole space \cite{MR4361973,MR4436794,MR4744252}.

In contrast to the previous works, the randomization in this paper is introduced before any spatial or temporal discretization is chosen. Consequently, the RBM parameter $h$ is independent of the numerical mesh and of the time integrator used later for implementation. The convergence result therefore concerns the randomized dynamics at the PDE level itself, rather than a particular fully discrete scheme. This separation is essential for the numerical experiments, where the same RBM construction can be combined with different time discretization methods without changing the underlying convergence mechanism. Furthermore, we focus on PDEs posed on metric graphs, providing a general framework to decompose the graph into active subgraphs and to formulate the corresponding randomized dynamics at the continuous level. It is important to note that the convergence estimates and frameworks from \cite{eisenmann2022randomized}, \cite{Corella_hernandez2025}, and \cite{firstpaper} do not imply the present continuous RBM setting, since their formulations do not decouple the RBM parameter $h$ from the respective discretization parameters. Consequently, our main result cannot be deduced from those prior works.

\subsection{Outline}

The remainder of this article is organized as follows.  In \cref{sec:overview}, we present the domain decomposition setting and the continuous RBM dynamics. Subsequently, we state the main results.  In \cref{sec:example_decomp}, we show a decomposition example for a particular graph. In \cref{sec:random_derivation}, we derive the continuous RBM system. Sections \ref{sec:well-posedness} and \ref{sec:convergence} contain proofs of the well‐posedness and convergence results, respectively. Finally, in \cref{sec:numerical_simulations}, we present numerical experiments that corroborate our theoretical convergence rate, the computational savings per RBM realization, and the effect of empirical averaging on sampling variability and computational time.

\section{Overview of the decomposition scheme and main results}\label{sec:overview}

In this section,  we present the RBM formulation, which leads to the definition of the random system. Then, we present the well‐posedness and convergence of the continuous RBM.

\subsection{Continuous random batch method}\label{sec:introduction_rbm}
The purpose of this subsection is to introduce the notation needed to describe which parts of the graph are active during each RBM time window. The construction has two levels: first, the graph is decomposed into elementary subgraphs. Then, these elementary subgraphs are grouped into batches that may be activated at random. The localization functions introduced below encode this activation directly at the PDE level.

We begin by introducing the graph decomposition, then the probabilistic setting, and finally, the random dynamics. Denote by $[M]=\{1,\dots,M\}$. An example of the following decomposition is provided in \Cref{sec:example_decomp}.
\medbreak

\noindent {\bf Graph decomposition.} Let $M\in \N$ and consider a decomposition of $\mathcal{G}$ via a family of subgraphs $\{\mathcal{G}_s^i\}_{i\in[M]} := \{(\mathcal{E}_s^i, \mathcal{V}_s^i)\}_{i\in[M]}$ that is pairwise edge-disjoint and covers $\mathcal{G}$, that is, 
\begin{align*}
    \mathcal{E} = \bigcup_{i\in [M]} \mathcal{E}_s^i,\quad \mathcal{V} = \bigcup_{i\in [M]} \mathcal{V}_s^i, \quad \text{and}\quad \mathcal{E}_s^i\cap \mathcal{E}_s^j = \emptyset,\quad\text{for }i\neq j. 
\end{align*}
In the following, we denote by $\mathcal{V}_{s,0}^i$ and $\mathcal{V}_{s,b}^i$ the interior and boundary nodes of $\mathcal{V}_s^i$, respectively. For every $v\in \mathcal{V}_{s,0}^i$, we denote by $\mathcal E_s^i(v):=\mathcal E_s^i\cap\mathcal E(v)$. We assume that $\mathcal{E}^i_s(v)=\mathcal{E}(v)$ for every $v\in \mathcal{V}_{s,0}^i$, that is, interior vertices in the subgraphs keep all adjacent edges.
\begin{figure}
\centering
\includegraphics[width=0.85\linewidth]{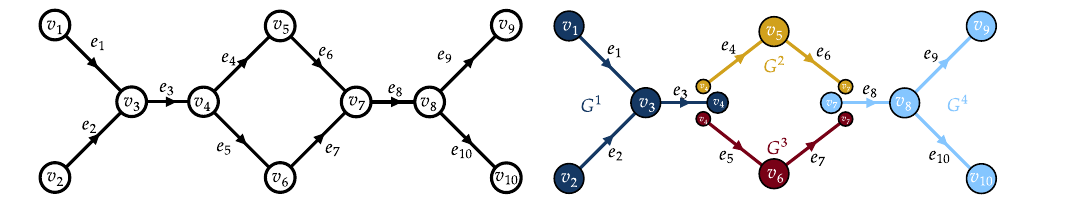}
\caption{Illustration of the graph $\mathcal{G}$ and its decomposition. (a) The full graph $\mathcal{G}$ with vertices $\mathcal{V} = \{v_1, \dots, v_{10}\}$ and edges $\mathcal{E} = \{e_1, \dots, e_{10}\}$. (b) Decomposition of $\mathcal{G}$ into non-overlapping subgraphs $\{\mathcal{G}_s^i\}_{i\in[4]}$.}\label{fig_1_deamon_graph}
\end{figure}

Denote by $ \mathcal{P}([M])$ the power set of $[M]$.Let $1\leq N\leq 2^M$ and let $\mathcal B_1,\dots,\mathcal B_N\in\mathcal P([M])\setminus\{\emptyset\}$ be such that $\bigcup_{j\in[N]}\mathcal B_j=[M]$. Then, for every $j\in [N]$, consider the set of edges and vertices:
\begin{align}\label{eq:random_graph_intro}
  \mathcal{E}^j := \bigcup_{i\in\mathcal B_{j}}\mathcal{E}_s^i,
  \hspace{0.25cm}
  \mathcal{V}^j := \bigcup_{i\in\mathcal B_{j}}\mathcal{V}_s^i, \hspace{0.25cm}
  \mathcal{V}_0^j:=\{v\in\mathcal{V}^j : \mathcal{E}^j(v)=\mathcal{E}(v)\},
  \hspace{0.25cm}
  \mathcal{V}_b^j:=\mathcal{V}^j\setminus\mathcal{V}_0^j,
\end{align}
where $\mathcal{E}^j(v):=\mathcal{E}^j\cap \mathcal{E}(v)$. Hence, $\mathcal{V}_0^j$ collects the vertices whose incident edges are all in $\mathcal{E}^j$, while $\mathcal{V}_b^j$ contains the remaining (interface and boundary) vertices; see \Cref{sec:example_decomp} and \Cref{fig:differnt_nodes}. Then, we have the following structural assumption for the batches.
\begin{assumption}\label{assump:A1}
    For every $v\in \mathcal{V}_0$ there exists a batch $\mathcal{B}_\star$ such that $v\in \mathcal{V}^\star_0$.
\end{assumption}
\smallskip\noindent\textbf{Localization function.}
Given $\psi:\mathcal{G}\to\R$ set
\begin{align}\label{eq:def_Si}
   \chi_\psi^{i,j}(x):=\psi(x)\,\mathbf 1_{S^i_j}(x), \quad  S^i_j:=\mathcal{G}^i_s\setminus(\mathcal{V}_{b}^j\cap\mathcal{V}_0),
\end{align}
for every $i\in[M]$ and $ j\in[N]$, where $\mathbf 1_{S^i_j}$ is the indicator function on $S^i_j$. By definition $\operatorname{supp}\chi_\psi^{i,j}\subset\mathcal{G}_s^i$ for every $i\in[M]$, and $\sum_{i\in [M]}\chi_\psi^{i,j}=\psi$ for every $j\in [N]$ and $ e\in\mathcal{E}$.
\medbreak

\noindent
{\bf Probabilistic framework:} Let $K \in \N$ and consider the RBM parameter $h = T/(K + 1) > 0$. Define the time grid $t_k:=k h$ for $k\in\{0,\dots,K\}$, and the intervals $I_k:=(t_{k-1},t_k)$ for each $k\in[K]$. Now, we introduce a family of i.i.d. random variables $\{\omega_k\}_{k\in[K]}$, each taking values in $[N]$. We denote by $p_j \in (0,1)$ the probability that $\omega_k$ equals $j \in [N]$, for each $k \in [K]$, which satisfy that $\sum_{j\in [N]} p_j = 1$. 
\medbreak

Then, for a given function $\psi : \mathcal{G} \to \R$, we consider the piecewise constant in time random function  
\begin{align}\label{eq:coef_random}
    \zeta^{\omega_k}_\psi(x): = \sum_{i \in \mathcal{B}_{\omega_k}} \frac{\chi^{i,\omega_k}_\psi(x)}{\pi_i}, \hspace{0.3cm} t \in I_k, \hspace{0.3cm}\text{with} \hspace{0.3cm}  \pi_{i} = \sum_{j \in [N] : i \in \mathcal{B}_j} p_j,  \hspace{0.3cm} \text{for all } i \in [M],
\end{align} 
for every $k \in [K]$, where $\pi_i$ is a normalization constant. We assume that $\pi_i>0$ for every $i\in[M]$. It is noteworthy that since $S_j^i$ and $\mathcal G_s^i$ differ only by vertices, their indicators agree for a.a. $x$ on each edge. Thus for a.a. $x$, we have
\begin{align}\label{eq:general_unbias}
\mathbb{E}[\zeta_\psi^{\omega_k}(x)]
&=\sum_{j\in [N]} \sum_{i\in\mathcal B_j}\frac{\psi(x)\,\mathbf1_{ S_j^i}(x)}{\pi_i} p_j
=\sum_{i\in [M]} \frac{\psi(x)}{\pi_i}\,\mathbf1_{\mathcal G_s^i}(x)\sum_{j\in [N]:\,i\in\mathcal B_j}p_j\\
&\nonumber=\sum_{i\in [M]} \psi(x)\,\mathbf1_{\mathcal G_s^i}(x)
=\psi(x).
\end{align}
Therefore, the random function $\zeta_\psi^{\omega_k}$ is an unbiased estimator of $\psi$ on the edges for every $k \in [K]$. For $t\in(0,T)$ let us introduce
\begin{align}\label{eq:coef_random2}
  k_t := \min\{k\in[K] : t < k h\},
  \quad\text{and}\quad 
  \zeta_\psi(x,t) := \sum_{i\in\mathcal B_{\omega_{k_t}}}\frac{\chi_\psi^{i,\omega_{k_t}}(x)}{\pi_i}.
\end{align}
Note that  $\zeta_\psi(\cdot,t) = \zeta_\psi^{\omega_k}$ for every $t\in I_k$, thus $\zeta_\psi(\cdot,t)$ is piecewise constant in time.
\smallbreak

\noindent
{\bf Random system:} Now, for every $e\in\mathcal{E}$, the continuous RBM system reads as:
\begin{align}\label{eq:heat1random}
\begin{cases}
\partial_t z_e - \partial_x(\zeta_a\partial_x z_e) + \zeta_b\partial_x z_e + \zeta_p z_e = \zeta_1 f_e, \hspace{-0.1cm} & (x,t)\in e\times(0,T), \\
z_{e'}(v,t)=z_{e''}(v,t), & (v,t)\in\mathcal{V}_0\times(0,T), e',e''\in\mathcal{E}(v), \\
\zeta_1(v,t)\bigl(z_e(v,t)-g_v(t)\bigr)=0, & (v,t)\in\mathcal{V}_b\times(0,T),e\in\mathcal{E}(v), \\
\sum_{e\in\mathcal{E}(v)}\zeta_a(v,t)\partial_x z_e(v,t)n_e(v)=0, & (v,t)\in\mathcal{V}_0\times(0,T), \\
z_e(x,0)=y_e^0(x), & x\in e.
\end{cases}
\end{align}
Note that the vertex–continuity condition in \eqref{eq:heat1random} is enforced on the whole graph, whereas the boundary and flux terms are switched on or off by the functions $\zeta_1$ and $\zeta_a$, depending on the subgraph that is active. A rigorous derivation of~\eqref{eq:heat1random}—obtained by alternately \emph{freezing} the state on an inactive subgraph and \emph{evolving} it on the active subgraph at each interval $I_k=(t_{k-1},t_k)$—is given in Section \ref{sec:random_derivation}.

\begin{remark}[On \Cref{assump:A1}]
\Cref{assump:A1} ensures that every interior vertex in $\mathcal{V}_0$ has a positive probability of appearing as an interior vertex of some randomly chosen subgraph. Consequently, although the Kirchhoff flux condition is imposed only at the active interior vertices $\mathcal V_0^k$ within each interval, it is enforced in expectation across all $v\in\mathcal V_0$, thereby guaranteeing the consistency of the continuous RBM scheme.
\end{remark}

\subsection{Main results}

In the following, we consider the standard Lebesgue and Sobolev spaces on the graph; see \Cref{appendix1}.
Regarding the well-posedness of \eqref{eq:heat1}, we have the following theorem, whose proof is given in \Cref{sec:well-posedness}. 
\begin{theorem}\label{th:well_posedness_nonhomo}
    Let the coefficients $(a,b,p)\in W^{1,\infty}_{pw}(\mathcal{E})\times L^\infty(\mathcal{E})\times L^\infty(\mathcal{E})$ with 
    \begin{align}\label{eq:assumption_coef_a}
        a_e(x)>a_0>0 \quad \text{for a.a.  } x\in e,\quad \text{for every } e\in\mathcal{E}.
    \end{align}
    Consider $y^0 \in L^2(\mathcal{E})$, $g \in H^1(0, T; l^2(\mathcal{V}_b))$, $f \in L^2(0, T; L^2(\mathcal{E}))$. Then, there exists a unique solution $y \in C([0, T]; L^2(\mathcal{E})) \cap L^2(0, T; H^1(\mathcal{E}))$ of \eqref{eq:heat1}. Moreover, if $y^0 \in H^1(\mathcal{E})$ and $g \in H^1(0, T; l^2(\mathcal{V}_b))$ satisfies the compatibility condition $y_e^0(v) = g_v(0)$ for each $v \in \mathcal{V}_b$, and the coefficient $b$ satisfies   
    \begin{align}\label{eq:assumption_coef_b}
        \sum_{e\in \mathcal{E}(v)}b_e(v)n_e(v) = 0, \quad\text{ for every }v\in \mathcal{V}_0,
    \end{align}
    then the unique solution of \eqref{eq:heat1} belongs to $C([0, T]; H^1(\mathcal{E}))\cap L^2(0,T; H^2(\mathcal{E}))\cap H^1(0,T; L^2(\mathcal{E}))$, and satisfies the energy estimate \begin{align}\label{eq:energy_estimation_y_wellposedness}
    \nonumber\|y\|_{H^1(0,T;L^2(\mathcal{E}))} +&\|y\|_{L^2(0,T;H^2(\mathcal{E}))} \\
    &\leq C \left( \|g\|_{H^1(0, T;l^2(\mathcal{V}_b))}+\|f\|_{L^2(0, T;L^2(\mathcal{E}))} + \|y^0\|_{H^1(\mathcal{E})} \right),
    \end{align}
    with $C$ a positive constant depending on $T$ and the coefficients $a,\,b$ and $p$.
\end{theorem}
\begin{remark}
    \Cref{th:well_posedness_nonhomo} establishes a maximal‐regularity result for parabolic PDEs on graphs. To our knowledge, this maximal regularity has not been established previously. The proof proceeds in two steps. First, one shows that the homogeneous (non–self‐adjoint) operator $A=-\partial_x(a\partial_x\,\cdot)+b\partial_x+p$ generates an analytic semigroup. Second, one identifies $D(A^{1/2})$ with a suitable interpolation space and then proves that this interpolation space coincides with $H_0^1(\mathcal{E})$.
\end{remark}
The following theorem addresses the well-posedness of the random system \eqref{eq:heat1random}.

\begin{theorem}\label{th:wellpossednes_random}
    Let $y^0 \in H^1(\mathcal{E})$, $f \in L^2(0, T; L^2(\mathcal{E}))$, the coefficients $(a, b, p) \in W^{1,\infty}_{pw}(\mathcal{E}) \times L^\infty(\mathcal{E}) \times L^\infty(\mathcal{E})$ satisfying \eqref{eq:assumption_coef_a}-\eqref{eq:assumption_coef_b}, and $g \in H^1(0, T; l^2(\mathcal{V}_b))$ satisfy the compatibility condition $y^0_e(v) = g_v(0)$ for each $v \in \mathcal{V}_b$. Then, there exists a unique solution $z$ of \eqref{eq:heat1random} belonging to the class $H^{1}(0, T; L^2(\mathcal{E})) \cap C([0, T]; H^1(\mathcal{E}))$, and the following energy estimate holds:
    \begin{align}\label{eq:energy_heat}
       \nonumber \|z\|_{H^1(0, T; L^2(\mathcal{E}))} + &\|z\|_{L^2(0, T; H^1(\mathcal{E}))} \\&\leq C \left( \|g\|_{H^1(0, T; l^2(\mathcal{V}_b))} + \|f\|_{L^2(0, T; L^2(\mathcal{E}))}+ \|y^0\|_{H^1(\mathcal{E})} \right),
    \end{align}
    where $C > 0$  is a constant depending on the coefficients and $T$, but independent of $h$. 
\end{theorem}

\begin{remark}
The proof of \Cref{th:wellpossednes_random} hinges on two facts. (i) By \Cref{th:well_posedness_nonhomo}, on each interval $I_k$, the dynamics restricted to the active subgraph admits a unique solution whose terminal value lies in $H^{1}$ on that subgraph. (ii)  The vertex–continuity condition in \eqref{eq:heat1random} propagates this $H^{1}$ regularity to the full graph.  
Starting with $z_0\in H^{1}(\mathcal{E})$, we therefore obtain $z(t_1)\in H^{1}(\mathcal{E})$ after the first step, and the argument can be re-applied on $I_{2}$. Induction over $k\in[K]$ yields global well-posedness.
\end{remark}

\smallbreak

To quantify the variance of the RBM, we introduce $\Lambda:(0,T)\to\R$ by
\begin{align}\label{eq:def_lambda}
     \Lambda(t) :=&
    \mathbb{E}\left[\|\partial_{x}((a - \zeta_a) \partial_x y)\|_{L^2(\mathcal{E})}^2 + \| (b - \zeta_b) \partial_x y\|_{L^2(\mathcal{E})}^2\right.\\
    &\nonumber\qquad\left.+ \|(p - \zeta_p) y\|_{L^2(\mathcal{E})}^2 + \|(1-\zeta_1 ) f\|_{L^2(\mathcal{E})}^2\right],
\end{align}
where $y$ is the solution of \eqref{eq:heat1}, and $f$ and $(a,b,p)$ are the respective source term and coefficients. Note that when $N = 1$,  systems \eqref{eq:heat1} and \eqref{eq:heat1random} coincide and $\Lambda \equiv 0$.

Now, we present our main convergence theorem.

\begin{theorem}\label{th:convergence}
Let $y^0 \in H^1(\mathcal{E})$, $f \in L^2(0, T; L^2(\mathcal{E}))$, the coefficients $(a, b, p) \in W^{1,\infty}_{pw}(\mathcal{E}) \times L^\infty(\mathcal{E}) \times L^\infty(\mathcal{E})$ satisfying \eqref{eq:assumption_coef_a}-\eqref{eq:assumption_coef_b}, and $g \in H^1(0, T; l^2(\mathcal{V}_b))$ satisfy the compatibility condition $y^0_e(v) = g_v(0)$ for each $v \in \mathcal{V}_b$. Suppose that Assumption \ref{assump:A1} holds. Let $y$ be the solution of \eqref{eq:heat1} and $z$ the solution of \eqref{eq:heat1random}. Then, $\Lambda\in L^1(0,T)$ is uniformly bounded with respect to $h$, and
    \begin{align*}
        \sup_{t \in [0, T]} \mathbb{E}\big[\|z(t) - y(t)\|^2_{L^2(\mathcal{E})}\big] \leq C \|\Lambda\|_{L^1(0, T)}^{1/2} h,
    \end{align*}
    where $C > 0$ depends on the problem data and $T$, but is independent of $h$.
\end{theorem}

\begin{remark}
  Note that the regularity assumptions in \Cref{th:convergence} are the same as those in \Cref{th:well_posedness_nonhomo} and \Cref{th:wellpossednes_random} when the maximal regularity is established. This is not merely a coincidence; the proof of \Cref{th:convergence} relies on energy estimates, for which $H^1(0,T;L^2(\mathcal{E}))$ and $L^2(0,T;H^2_{pw}(\mathcal{E}))$ regularity of $z$ and $y$, respectively, is required. Then, after several energy estimates, and ensuring constants independent of $h$, the proof is concluded by an appropriate application of Gr\"onwall's lemma.
\end{remark}
Finally, the following corollary is a direct consequence of Theorem \ref{th:convergence}.

\begin{corollary}\label{coro:convergence}Let $y^0 \in H^1(\mathcal{E})$, $f \in L^2(0, T; L^2(\mathcal{E}))$, the coefficients $(a, b, p) \in W^{1,\infty}_{pw}(\mathcal{E}) \times L^\infty(\mathcal{E}) \times L^\infty(\mathcal{E})$ satisfying \eqref{eq:assumption_coef_a}-\eqref{eq:assumption_coef_b}, and $g \in H^1(0, T; l^2(\mathcal{V}_b))$ satisfy the compatibility condition $y^0_e(v) = g_v(0)$ for each $v \in \mathcal{V}_b$. Let $y$ be the solution of \eqref{eq:heat1} and $z$ the solution of \eqref{eq:heat1random}. Then, the following statements hold:
    \begin{itemize}
        \item[(i)] There exists a constant $C > 0$, independent of $h > 0$, such that 
        \begin{align*}
            \sup_{t \in [0, T]} \|\mathbb{E}[z(t)] - y(t)\|_{L^2(\mathcal{E})}^2 \leq C \|\Lambda\|_{L^1(0, T)}^{1/2} h.
        \end{align*}
        
        \item[(ii)] There exists a constant $C > 0$, independent of $h > 0$, such that 
        \begin{align*}
            \sup_{t \in [0, T]} \Var\big[\|z(t) - y(t)\|_{L^2(\mathcal{E})}\big] \leq C \|\Lambda\|_{L^1(0, T)}^{1/2} h.
        \end{align*}
        
        \item[(iii)] For all $\varepsilon > 0$, there exists a constant $C > 0$, independent of $h > 0$, such that 
        \begin{align*}
            \sup_{t \in [0, T]}\mathbb{P}\bigg( \|z(t) - y(t)\|_{L^2(\mathcal{E})} > \varepsilon\bigg) \leq \frac{C \|\Lambda\|_{L^1(0, T)}^{1/2} h}{\varepsilon^2}.
        \end{align*}
    \end{itemize}
\end{corollary}

\begin{proof}
    Part (i) follows from Theorem \ref{th:convergence} and Jensen's inequality. Part (ii) is derived using the relation $\Var[X] = \mathbb{E}[X^2] - (\mathbb{E}[X])^2$ along with part (i). Finally, part (iii) is a consequence of Markov's inequality applied to Theorem \ref{th:convergence}.
\end{proof}

\Cref{coro:convergence} ensures that the expected trajectory of the continuous RBM approaches the actual trajectory as $h \to 0$. Additionally, as $h$ decreases, the variance of the RBM error diminishes. This is particularly significant because it implies that for sufficiently small $h$, a single realization of the random system serves as a representative trajectory of the expected trajectory, thereby closely approximating the solution of the parabolic equation.

\section{A decomposition example}\label{sec:example_decomp}

We now illustrate a non-overlapping decomposition of $\mathcal{G}$ following \Cref{sec:introduction_rbm}.

\smallbreak
\noindent {\bf 1. Feasible decomposition:} 
We will consider the graph and decomposition introduced in \Cref{fig_1_deamon_graph}.  By inspection, we note that the decomposition is pairwise edge-disjoint. Moreover, for each interior vertex $v\in\mathcal{V}_{s,0}^i$ one checks that $\mathcal{E}^i_s(v)=\mathcal{E}(v)$, for every $i\in[4]$. For example, at $v_3\in\mathcal{V}_{s,0}^1$ we have  $\mathcal{E}^1_s(v_3) = \{e_1,e_2,e_3\} = \mathcal{E}(v_3)$. Then, within this decomposition, we introduce the following family of batches
\begin{itemize}
    \item {\bf Option 1:} $\mathcal{B}_{1} = \{1\}$, $\mathcal{B}_{2} = \{2\}$, $\mathcal{B}_{3} = \{3\}$, $\mathcal{B}_{4} = \{4\}$, $\mathcal{B}_{5} = \{1,2,3\}$, and $\mathcal{B}_{6} = \{2,3,4\}$; or
    \item {\bf Option 2:} $\mathcal{B}_{1} = \{1\}$, $\mathcal{B}_{2} = \{2\}$, $\mathcal{B}_{3} = \{3\}$, $\mathcal{B}_{4} = \{4\}$, $\mathcal{B}_5=[4]$.
\end{itemize}
Note that neither $v_4$ nor $v_7$ lies in any $\mathcal{V}_{s,0}^i$. However, both {\bf option 1} and  {\bf option 2}, guarantee \Cref{assump:A1}. In fact, for {\bf option 1} we have that $v_4\in \mathcal{V}^5$ and $v_7\in \mathcal{V}^6$ (see table below), while for {\bf option 2} we have that $v_4,\,v_7\in \mathcal{V}^5$. Note that, as long as there exists $\mathcal{B}_* =[M]$, \cref{assump:A1} automatically holds.
\medbreak

\noindent {\bf 2. Localization function:} Let us first analyze the sets of vertices defined in \eqref{eq:random_graph_intro}, associated with the batches from {\bf option 1}. We have the following table:
\vspace{-0.2cm}

\begin{table}[H]
\centering
\resizebox{\textwidth}{!}{%
\begin{tabular}{c|cccccc}
 & $\mathcal{B}_1$ & $\mathcal{B}_2$ & $\mathcal{B}_3$ & $\mathcal{B}_4$ & $\mathcal{B}_5$ & $\mathcal{B}_6$ \\
\hline
$\mathcal{V}^j$   & $\{v_1,v_2,v_3,v_4\}$                    & $\{v_4,v_5,v_7\}$                        & $\{v_4,v_6,v_7\}$                        & $\{v_7,v_8,v_9,v_{10}\}$                 & $\{v_1,v_2,v_3,v_4,v_5,v_6,v_7\}$            & $\{v_4,v_5,v_6,v_7,v_8,v_9,v_{10}\}$      \\
$\mathcal{V}_0^j$ & $\{v_3\}$                                & $\{v_5\}$                                & $\{v_6\}$                                & $\{v_8\}$                                & $\{v_3,v_4,v_5,v_6\}$                        & $\{v_4,v_5,v_6,v_7,v_8\}$                    \\
$\mathcal{V}_b^j$ & $\{v_1,v_2,v_4\}$                        & $\{v_4,v_7\}$                            & $\{v_4,v_7\}$                            & $\{v_7,v_9,v_{10}\}$                     & $\{v_1,v_2,v_7\}$                            & $\{v_3,v_9,v_{10}\}$                         \\
$\mathcal{V}_b^j\cap\mathcal{V}_0$ & $\{v_4\}$ & $\{v_4,v_7\}$ & $\{v_4,v_7\}$ & $\{v_7\}$ & $\{v_7\}$ & $\{v_3\}$ \\
\end{tabular}%
}
\end{table}
\vspace{-0.3cm}
Now, we can analyze the function $\zeta^j_\psi$ defined in \eqref{eq:coef_random} for $j\in[6]$. By definition:
\begin{equation}\label{eq:example_zetas}
\begin{aligned}
        &\zeta^5_\psi(x) \! =  \!\sum_{i\in \mathcal B_5}\frac{\chi_\psi^{i,5}}{\pi_i}  \!= \!\frac{\chi_\psi^{1,5}(x)}{\pi_1}+\frac{\chi_\psi^{2,5}(x)}{\pi_2}+\frac{\chi_\psi^{3,5}(x)}{\pi_3}, \\
     &\zeta^6_\psi(x) \!= \! \sum_{i\in  \mathcal B_6}\frac{\chi_\psi^{i,6}}{\pi_i}  \!= \!\frac{\chi_\psi^{2,6}(x)}{\pi_2}+\frac{\chi_\psi^{3,6}(x)}{\pi_3}+\frac{\chi_\psi^{4,6}(x)}{\pi_4}, \quad\zeta^j_\psi(x)  \!=  \!\sum_{i\in  \mathcal B_j}\frac{\chi_\psi^{i,j}}{\pi_i}  \!= \!\frac{\chi_\psi^{j,j}(x)}{\pi_j},
\end{aligned}
\end{equation}
for $j\in[4]$. Figure \ref{fig:chi_functions} illustrates $\zeta^1_\psi$, $\zeta^5_\psi$, and $\zeta^6_\psi$ when $\psi\equiv1$, and $p_j=1/6$ for $j\in [6]$. In this setting, we have that 
\begin{align*}
    \pi_1 = \frac{1}{6}+\frac{1}{6} =\frac{1}{3},\quad \pi_2 = \frac{1}{6}+\frac{1}{6}+\frac{1}{6} =\frac{1}{2},\quad
    \pi_3 = \frac{1}{6}+\frac{1}{6}+\frac{1}{6} =\frac{1}{2},\quad \pi_4 = \frac{1}{6}+\frac{1}{6} =\frac{1}{3}.
\end{align*}
\begin{figure}[h!]
    \centering
    \includegraphics[width=\linewidth]{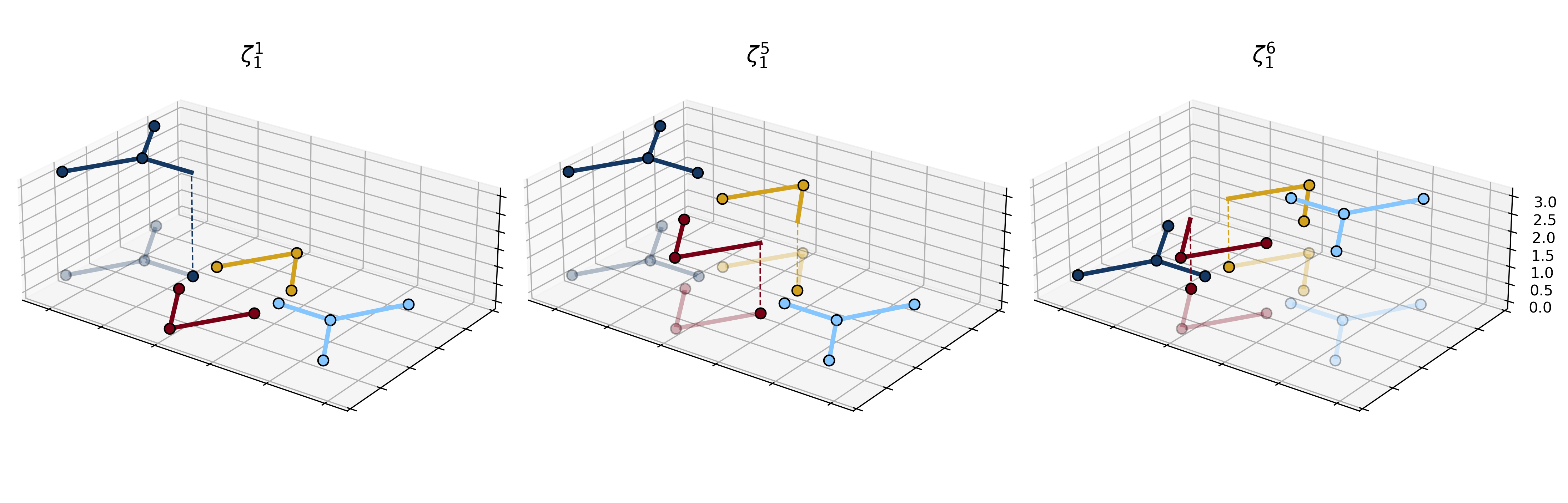}
    \caption{Illustration of the functions $\zeta^1_1$, $\zeta^5_1$, and $\zeta^6_1$ defined in \eqref{eq:example_zetas}. Here, each edge and vertex is identified as in \Cref{fig_1_deamon_graph}. The z-axis for the three plots varies between 0 and 3.}
    \label{fig:chi_functions}
\end{figure}
Note that $\zeta^j_1$ can be discontinuous at $\mathcal{V}_0$. Thus, even if $\psi$ is continuous across the vertices, $\zeta^j_1$ might not be. In general, if $\psi\in W^{1,p}_{pw}(\mathcal{E})$, then $\zeta^j_\psi\in W^{1,p}_{pw}(\mathcal{E})$ for every $j\in[N]$ and $p\in [1,\infty]$. 

\section{Derivation of the random system}\label{sec:random_derivation}
This section explains how the compact random system \eqref{eq:heat1random} is obtained from the freeze-and-evolve procedure. The auxiliary variables $\xi^k$ and $\eta^k$ describe one RBM time window: $\xi^k$ represents the frozen state on the inactive edges, whereas $\eta^k$ represents the solution evolved on the active subgraph. The global RBM trajectory is obtained by gluing these two pieces at the end of each time window.

Let us derive the random system \eqref{eq:heat1random}. Initially, the system is defined separately for each time interval. In the following, for simplicity, for every $k\in [K]$ we denote by $\zeta^k = \zeta^{\omega_k}$ and $\mathcal{G}^k=(\mathcal{E}^k,\mathcal{V}^k) = (\mathcal{E}^{\omega_k},\mathcal{V}^{\omega_k})=\mathcal{G}^{\omega_k}$. That is, $\mathcal{G}^k$ is the randomly chosen graph, in the $k$-th time interval, whose edges and vertices are defined as in \eqref{eq:random_graph_intro}. Now, for each $k \in [K]$, we solve the ODE 
\begin{align}\label{eq:random_interval2}
    \begin{cases}
        \partial_t \xi_e^k = 0, & (x,t) \in e \times I_k, \, e \in \left( \mathcal{E}^k \right)^c, \\
        \xi_{e'}^k(v,t) = \xi_{e''}^{k}(v,t), &v\in \left( (\mathcal{V}^k)^c \setminus \mathcal{V}_b \right),\,\forall e',\, e'' \in \mathcal{E}(v), \\
        \xi_e^k(x,t_{k-1}) = z_e^{k-1}(x,t_{k-1}), & x \in e \in \left( \mathcal{E}^k \right)^c,
    \end{cases}
\end{align}
for every $t\in I_k$. Here, the function $z_e^{k}$ is defined by
\begin{align}\label{eq:definition_z}
      z_e^{k} = \begin{cases}
          \xi_e^{k}(x,t_{k}), & \text{in } \left( \mathcal{E}^k \right)^c, \\
          \eta_e^{k}(x,t_{k}), & \text{in } \mathcal{E}^k,
      \end{cases}
\end{align}
for every $k \in [K]$, with $z^0_e(x,0) = y_e^0(x)$, and $\eta$ is the solution of the parabolic equation
\begin{align}\label{eq:random_interval1}
      \begin{cases}
        \partial_t \eta_e^k - \partial_{x}(\zeta^k_a \partial_x \eta^k_e) + \zeta_b^k \partial_x \eta_e^k + \zeta_p^k \eta_e^k = \zeta^k_1 f_e, & (x,t) \in e \times I_k, \, e \in \mathcal{E}^k, \\
        \eta_{e'}^k(v,t) = \eta_{e''}^k(v,t), & v \in \mathcal{V}_0^k,\, \forall e',\, e'' \in \mathcal{E}(v),\\
        \eta_{e'}^k(v,t) = \xi_{e''}^k(v,t), & v \in (\mathcal{V}^k_b \cap \mathcal{V}_0),\, \forall e', e'' \in \mathcal{E}(v), \\
        \eta_e^k(v,t) = g_v(t), & v \in (\mathcal{V}_b^k \cap \mathcal{V}_b),\, e \in \mathcal{E}(v), \\
        \sum_{e \in \mathcal{E}(v)}\zeta^k_a(v)\partial_x \eta_e^k(v) n_e(v) = 0, &v \in \mathcal{V}_0^k, \\
        \eta_e^k(x,t_{k-1}) = z^{k-1}(x,t_{k-1}), & x \in e \in \mathcal{E}^k,
      \end{cases}
\end{align}
for every $t\in I_k$ and $k \in [K]$. Here $\mathcal{V}^k_b \cap \mathcal{V}_0$ are the vertices on the interface between the chosen subgraph and its complement (see \Cref{fig:differnt_nodes}). The second condition in \eqref{eq:random_interval1} corresponds to a continuity condition inside the chosen subgraph $\mathcal{G}^k$. The third and fourth conditions in \eqref{eq:random_interval1} correspond to the boundary data of the problem. The fifth condition is the flux condition inside $\mathcal{V}_0^k$. 

\begin{remark}\label{remark_system_discrete} On each interval $I_k$, equation \eqref{eq:random_interval2} simply freezes the solution from the previous time interval on the inactive edges by setting $  \xi^k_e \equiv z^{k-1}_e$ for $e\in(\mathcal{E}^k)^c,$ while \eqref{eq:random_interval1} evolves the solution on the active subgraph $\mathcal{G}^k$, imposing continuity across the interface vertices and the prescribed Dirichlet data on $\mathcal{V}_b^k\cap\mathcal{V}_b$. Since no flux condition is enforced at the interface (to avoid overdetermination of the system), the glued solution $z^k$ generally does not belong to $H^2(\mathcal E)$, although the continuity condition guarantees $z^k\in H^1(\mathcal E)$ at each step.
\end{remark}

\begin{figure}
      \centering
      \includegraphics[width=0.6\linewidth]{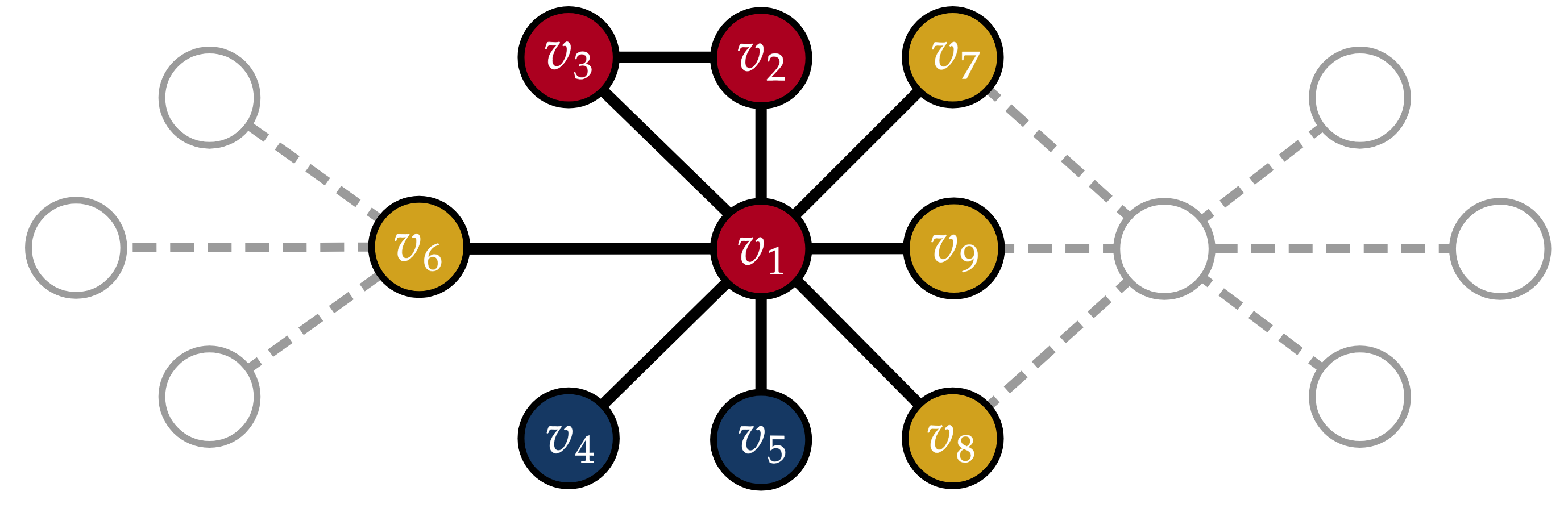}
      \caption{Illustration of a subgraph defined by \eqref{eq:random_graph_intro}. For a given batch $\mathcal{B}_j$, let $\mathcal{E}^j$ and $\mathcal{V}^j$ be the edges and vertices associated. We illustrate $\mathcal{E}^j$ with solid black lines. The colored vertices belong to $\mathcal{V}^j$. Red vertices $\{v_1,v_2,v_3\}=\mathcal{V}_0^j$; blue vertices $\{v_4,v_5\}=\mathcal{V}_b^j \cap \mathcal{V}_b$; yellow vertices (interface vertices) $\{v_6,v_7,v_8,v_9\}=\mathcal{V}_b^j \cap \mathcal{V}_0$; the blue and yellow vertices $\{v_4,v_5,v_6,v_7,v_8,v_9\}=\mathcal{V}_b^j$. Remaining unlabeled vertices are in $(\mathcal{V}^j)^c$.}
      \label{fig:differnt_nodes}
    \end{figure}

 One can verify that, on each $I_k$,  \eqref{eq:heat1random} on $(\mathcal{G}^k)^c$ reduces to \eqref{eq:random_interval2} and \eqref{eq:heat1random} on $\mathcal{G}^k$ reduces to \eqref{eq:random_interval1}. System \eqref{eq:heat1random} will be primarily used to establish the convergence of the method, while systems \eqref{eq:random_interval2}-\eqref{eq:random_interval1} are more suitable (yet equivalent) for numerical implementation. In particular, system \eqref{eq:random_interval2}-\eqref{eq:random_interval1} can be solved by following \Cref{alg1}.

\begin{algorithm}
  \caption{Subgraph Decomposition Scheme}\label{alg1}
  \begin{algorithmic}[1]
    \STATE{ Define $z^{0}\gets y^{0}$}
    \WHILE{$k+1 \leq K$}
      \STATE{ Sample $\omega_{k}$; set $\mathcal{G}^{\omega_k}=(\mathcal{E}^{\omega_k},\mathcal{V}^{\omega_k})$ by \eqref{eq:random_graph_intro}}
      \STATE{ Define $\xi^{k}\gets z^{\,k-1}$ on $\mathcal{E}\setminus\mathcal{E}^{\omega_k}$}
      \STATE{ Compute $(\zeta_a^k,\zeta_b^k,\zeta_p^k,\zeta_1^k)$ via \eqref{eq:coef_random2}}
      \STATE{ Solve \eqref{eq:random_interval1} on $\mathcal{G}^{\omega_k}$ for $\eta^{k}$ using $\xi^{k}$}
      \STATE{ Define $z^{k}_{e}\gets$ $\eta^{k}_{e}$ and $\xi^{k}_{e}$ via \eqref{eq:definition_z}}
      \STATE{ Define $k\gets k+1$}
      \ENDWHILE
  \end{algorithmic}
\end{algorithm}

\section{Well-posedness of the parabolic equation on a graph}\label{sec:well-posedness} The proof of \Cref{th:well_posedness_nonhomo} proceeds in two main steps. First, we show that the corresponding homogeneous system is well-posed by proving that its associated operator is m-accretive. By the Hille--Yosida theorem, this guarantees a unique solution in $C([0,T];L^2(\mathcal{E}))$. Second, we prove the maximal regularity by establishing that the operator associated with the homogeneous system is analytic. Hence, for any initial condition in the real interpolation space $(D(A),L^2(\mathcal{E}))_{1/2}$, there exists a unique solution in $H^1(0,T; L^2(\mathcal{E}))\cap L^2(0,T;D(A))$, where $D(A)$ is the operator domain. We then show that $(D(A),L^2(\mathcal{E}))_{1/2}\cong H_0^1(\mathcal{E})$. These technical results can be found in \Cref{appendix1}. With these results established, the existence and uniqueness of the non-homogeneous system are followed by lifting the boundary data to the right-hand side of the PDE.
\subsection{Well-posedness of the homogeneous system}
We proceed by showing the well-posedness of the homogeneous Dirichlet boundary problem:
\begin{align}\label{eq:homo_dirichlet}
    \begin{cases}
        \partial_t u_e - \partial_{x}( a_e \partial_x u_e) + b_e \partial_x u_e + p_e u_e = \vartheta_e, & (x,t) \in e \times (0, T),  \, e \in \mathcal{E},\\
        u_{e'}(v, t) = u_{e''}(v, t), &v \in \mathcal{V}_0, \forall e', e'' \in \mathcal{E}(v), \\
        u_e(v, t) = 0, & v\in \mathcal{V}_b, \, e \in \mathcal{E}(v), \\
        \sum_{e \in \mathcal{E}(v)} a_e(v)\partial_x u_e(v, t) n_e(v) = 0, & v \in \mathcal{V}_0, \\
        u_e(x, 0) = u_e^0(x), & x \in e,
    \end{cases}
\end{align}
for every $t\in(0,T)$, where the coefficients $(a, b, p) \in W^{1,\infty}_{pw}(\mathcal{E}) \times L^{\infty}(\mathcal{E}) \times L^{\infty}(\mathcal{E})$ satisfy \eqref{eq:assumption_coef_a} and \eqref{eq:assumption_coef_b}. Then, the following theorem holds.
\begin{lemma}\label{th:well_posed_heat_graph}
    Let $u^0 \in L^2(\mathcal{E})$ and $\vartheta \in L^2(0, T; L^2(\mathcal{E}))$. Then, \eqref{eq:homo_dirichlet} has a unique solution $u \in C([0, T]; L^2(\mathcal{E}))$. Moreover, if $u^0 \in H_0^1(\mathcal{E})$, then \eqref{eq:homo_dirichlet} has a unique solution $ u \in C([0, T]; H_0^1(\mathcal{E})) \cap H^1(0, T; L^2(\mathcal{E})) \cap L^2(0, T; H^2(\mathcal{E})),$ and there exists a constant $C > 0$ such that the following energy estimate holds:
    \begin{align}\label{eq:u_estimation}
        \|u\|_{H^1(0, T; L^2(\mathcal{E}))} + \|u\|_{L^2(0, T; H^2(\mathcal{E}))} \leq C \left( \|\vartheta\|_{L^2(0, T; L^2(\mathcal{E}))} + \|u^0\|_{H^1_0(\mathcal{E})} \right).
    \end{align}
\end{lemma}

\begin{proof}
    Define the linear operator $A: D(A) \subset L^2(\mathcal{E}) \to L^2(\mathcal{E})$ by
    \begin{align}\label{def_A_heat}
    \begin{cases} Au := - \partial_x(a_e \partial_x u_e) + b_e \partial_x u_e + p_e u_e, \quad \text{for every } e \in \mathcal{E}, \\
    D(A) := \left\{ u \in H^1_0(\mathcal{E}) \cap H_{pw}^2(\mathcal{E}) \,\Big|\, \sum_{e \in \mathcal{E}(v)} a_e(v)\partial_x u_e(v) n_e(v) = 0, \, \forall v \in \mathcal{V}_0 \right\}.
        \end{cases}
    \end{align}
    By \Cref{prop:A_deg_generator}, $A$ is the infinitesimal generator of a strongly continuous semigroup on $L^2(\mathcal{E})$. Given $u^0 \in L^2(\mathcal{E})$, $(a, b, p) \in W^{1,\infty}_{pw}(\mathcal{E}) \times L^{\infty}(\mathcal{E}) \times L^{\infty}(\mathcal{E})$ satisfying \eqref{eq:assumption_coef_a}, and $\vartheta \in L^2(0, T; L^2(\mathcal{E}))$, \cite[Proposition 3.1]{Bensoussan} ensures the existence of a unique solution $u$ of \eqref{eq:homo_dirichlet} in $C([0, T]; L^2(\mathcal{E}))$.
    
    Furthermore, due to \Cref{th:maximal_regularity}, we have that if $u^0\in H_0^1(\mathcal{E})$ and the coefficient $b$ satisfies \eqref{eq:assumption_coef_b}, then the unique solution $u$ of \eqref{eq:homo_dirichlet} fulfills
    \begin{align*}
        u \in H^1(0, T; L^2(\mathcal{E})) \cap L^2(0, T; H^2(\mathcal{E})) \cap C([0, T]; H_0^1(\mathcal{E})),
    \end{align*}
    and the energy estimate \eqref{eq:u_estimation} holds, concluding the result.
\end{proof}

\subsection{Well-posedness of the non-homogeneous system}
\begin{proof}[Proof of \Cref{th:well_posedness_nonhomo}]
    Consider a quadratic function $\mu$ defined by $\mu_e(x,t) = \alpha_e(t) x^2 + \beta_e(t) x + \gamma_e(t),$ satisfying the boundary conditions $\mu_e(v,t) = \partial_x \mu_e(v,t) = 0$ for every $v \in \mathcal{V}_0$ and $ \mu_e(v,t) = g_v(t)$ for every $v \in \mathcal{V}_b$. As $g_v \in H^1(0,T)$, choosing $\alpha,\beta,\gamma \in H^1(0,T)$ yields $\mu \in H^1(0,T; H^2(\mathcal{E}))$. Define $u$ to be the solution to the homogeneous problem \eqref{eq:homo_dirichlet} with the modified source term $\vartheta_e := \partial_x(a_e \partial_x \mu_e) - b_e \partial_x \mu_e - p_e \mu_e + f_e,$ and the initial condition $ u^0_e(x) = y^0_e(x) - \mu_e(x,0)$. Since $(a, b, p) \in W^{1,\infty}_{pw}(\mathcal{E}) \times L^{\infty}(\mathcal{E}) \times L^{\infty}(\mathcal{E})$, we have that $\vartheta_e\in L^2(0,T; L^2(\mathcal{E}))$.
    By \Cref{th:well_posed_heat_graph}, there exists a unique solution $u \in C([0,T]; L^2(\mathcal{E}))$ to \eqref{eq:homo_dirichlet}. Moreover, if $y^0\in H^1(\mathcal E)$ and $y_e^0(v)=g_v(0)$ for every $v\in\mathcal V_b$, then $u^0=y^0-\mu(\cdot,0)$ belongs to $H_0^1(\mathcal E)$. Consequently,
    \begin{align*}
        u \in H^1(0,T; L^2(\mathcal{E})) \cap L^2(0,T; H^2(\mathcal{E})) \cap C([0,T]; H^1(\mathcal{E})),
    \end{align*}
    and the energy estimate
    \begin{align*}
        \|u\|_{H^1(0,T; L^2(\mathcal{E}))} &+ \|u\|_{L^2(0,T; H^2(\mathcal{E}))} \\
        &\leq C \left( \|g\|_{H^1(0,T; l^2(\mathcal{V}_b))} + \|f\|_{L^2(0,T; L^2(\mathcal{E}))} + \|u^0\|_{H^1(\mathcal{E})} \right),
    \end{align*}
    holds, where $C > 0$ is a constant independent of $h$. Define $y = u + \mu$, then $y$ satisfies \eqref{eq:heat1} with the initial condition $y_e(x,0) = u_e(x,0) + \mu_e(x,0) = y^0_e(x)$ and boundary conditions given by $g_v(t)$. For the uniqueness of $y$, let $y_1$ and $y_2$ be two solutions of \eqref{eq:heat1}, then $\tilde u:= y_1-y_2$ is solution of \eqref{eq:homo_dirichlet} with $\vartheta\equiv u_0\equiv 0$. By \eqref{eq:u_estimation}, we deduce the uniqueness. Moreover, using the regularity of $u$ and $\mu$, we obtain
    \begin{align*}
        y \in H^1(0,T; L^2(\mathcal{E})) \cap L^2(0,T; H^2(\mathcal{E})) \cap C([0,T]; H^1(\mathcal{E})).
    \end{align*}
    The energy estimate for $y$ follows from the estimates on $u$ and the properties of $\mu$. Specifically, applying \eqref{eq:u_estimation} and utilizing the continuity of the right-inverse of the trace operator in $H_0^1(\mathcal{E})$ (see \cite[Chapter I, Theorem 8.3]{lions_book}), we obtain
    \begin{align*}
        \|y\|_{H^1(0,T; L^2(\mathcal{E}))} +&\|y\|_{L^2(0,T; H^2(\mathcal{E}))}\\
        &\leq C \left( \|g\|_{H^1(0,T; l^2(\mathcal{V}_b))} + \|f\|_{L^2(0,T; L^2(\mathcal{E}))} + \|u^0\|_{H^1(\mathcal{E})} \right),
    \end{align*}
    where $C > 0$ is a constant independent of $h$.
\end{proof}

 \subsection{Well-posedness of the continuous RBM}\label{sec:well-posed-random}

Let us prove \Cref{th:wellpossednes_random}.

\begin{proof}[Proof of \Cref{th:wellpossednes_random}]
We prove the well-posedness by induction on $k \in [K]$.

\smallbreak

\noindent
\textbf{Base Case ($k = 1$):} Fix a subgraph $\mathcal{G}^1 = (\mathcal{E}^1, \mathcal{V}^1)$ selected at random. Since $y^0 \in H^1(\mathcal{E})$, it follows that $z^0 \in H^1\bigl((\mathcal{E}^1)^c\bigr)$. Hence, $\xi^1(x,t) = z^0(x)$ is the unique solution of \eqref{eq:random_interval2} on $I_1 \times (\mathcal{E}^1)^c$. Next, fixing  $\xi^1(v,t)$ as boundary data in $(\mathcal{V}_b^1 \cap \mathcal{V}_0)$ and $g_v(t)$ on $(\mathcal{V}_b^1 \cap \mathcal{V}_b)$, \Cref{th:well_posedness_nonhomo} ensures the existence of a unique solution
\begin{align*}
  \eta^1 \in C\bigl(\overline{I_1}; H^1(\mathcal{E}^1)\bigr) 
  \cap
  H^1\bigl(I_1; L^2(\mathcal{E}^1)\bigr)
  \cap
  L^2\bigl(I_1; H^2(\mathcal{E}^1)\bigr)
\end{align*}
to \eqref{eq:random_interval1}. Define
\begin{align*}
  z^1(x,t)
  =
  \begin{cases}
    \eta^1(x,t), & x \in \mathcal{E}^1, \\
    \xi^1(x,t),    & x \in \bigl(\mathcal{E}^1\bigr)^c.
  \end{cases}
\end{align*}
Then $z^1\in 
  C\bigl(\overline{I_1}; H^1(\mathcal{E})\bigr)
  \cap
  H^1\bigl(I_1; L^2(\mathcal{E})\bigr)$ is the unique solution of \eqref{eq:heat1random} on $I_1 \times \mathcal{E}$.

\smallbreak

\noindent
\textbf{Inductive step:} 
For some $n\in [K-1]$, suppose that there is a unique solution
\begin{align*}
  z \in 
  C\bigl(\overline{I_n}; H^1(\mathcal{E})\bigr)
  \cap
  H^1\bigl(I_n; L^2(\mathcal{E})\bigr)
  \cap
  L^2\bigl(I_n; H^1(\mathcal{E})\bigr)
\end{align*}
of \eqref{eq:heat1random} on $I_n \times \mathcal{E}$. We now prove the existence and uniqueness for $n+1$.

Pick a random subgraph $\mathcal{G}^{n+1} = (\mathcal{E}^{n+1}, \mathcal{V}^{n+1})$. Let $z^n(x,t_n) \in H^1(\mathcal{E})$ be the initial condition for \eqref{eq:random_interval1} and \eqref{eq:random_interval2} in their respective subgraphs. Then, $ \xi^{n+1}(x,t) = z^n(x,t_n)$ is the unique solution of \eqref{eq:random_interval2} in $I_{n+1} \times (\mathcal{E}^{n+1})^c$. Considering $\xi^{n+1}(v,t)$ as boundary data in $ \mathcal{V}_b^{n+1} \cap \mathcal{V}_0$ and $g_v(t)$ on $\mathcal{V}_b^{n+1} \cap \mathcal{V}_b$, \Cref{th:well_posedness_nonhomo} guarantees the existence of unique solution of \eqref{eq:random_interval1} in the class
\begin{align*}
  \eta^{n+1} 
  \in 
  C\bigl(\overline{I_{n+1}}; H^1(\mathcal{E}^{n+1})\bigr)
  \cap
  H^1\bigl(I_{n+1}; L^2(\mathcal{E}^{n+1})\bigr)
  \cap
  L^2\bigl(I_{n+1}; H^2(\mathcal{E}^{n+1})\bigr).
\end{align*}
Define
\begin{align*}
  z^{n+1}(x,t)
  =
  \begin{cases}
    \eta^{n+1}(x,t), & x \in \mathcal{E}^{n+1}, \\
    \xi^{n+1}(x,t),    & x \in \bigl(\mathcal{E}^{n+1}\bigr)^c.
  \end{cases}
\end{align*}
Then $z^{n+1}$ is the unique solution of \eqref{eq:heat1random} on $I_{n+1} \times \mathcal{E}$. By \Cref{th:well_posedness_nonhomo},
\begin{align*}
  z^{n+1} 
  \in 
  C\bigl(\overline{I_{n+1}}; H^1(\mathcal{E})\bigr)
  \cap
  H^1\bigl(I_{n+1}; L^2(\mathcal{E})\bigr).
\end{align*}
This completes the induction step. Thus, by induction, for each $k \in [K]$, there is a unique solution of \eqref{eq:heat1random} on $I_k \times \mathcal{E}$. Consequently, $ z(x,t) = z^{k_t}(x,t),$ corresponds to the unique solution of \eqref{eq:heat1random} in $(0,T) \times \mathcal{E}$. Furthermore, $z 
  \in
  C\bigl([0,T]; H^1(\mathcal{E})\bigr)
  \cap
  H^1\bigl(0,T; L^2(\mathcal{E})\bigr),$
and satisfies $z \in L^2\bigl(I_k; H^2(\mathcal{E}^k)\bigr)$ for each $k \in [K]$.  
\end{proof}

\section{Convergence of the random scheme}\label{sec:convergence}
In this section, we prove the convergence in expectation of the continuous RBM. The proof has two parts. First, we show that the random coefficients and source term have a uniformly bounded variance along the deterministic solution; this is the role of the quantity $\Lambda$. Second, we estimate the evolution of the error $z-y$ on each RBM time window and use the independence of the random choices to obtain an $O(h)$ bound. We first prove the bound on $\Lambda$.

\begin{lemma}\label{lemma:lambda_bounded}
Let $\Lambda$ be defined as in \eqref{eq:def_lambda}. Then, under the assumptions of \Cref{th:well_posedness_nonhomo}, we have $\Lambda \in L^1(0,T)$ and
    \begin{align*}
        \|\Lambda\|_{L^1(0, T)} \leq C \left( \|g\|_{H^1(0, T; l^2(\mathcal{V}_b))}^2 + \|f\|_{L^2(0, T; L^2(\mathcal{E}))}^2 + \|y^0\|_{H^1(\mathcal{E})}^2 \right),
    \end{align*}
    where $C > 0$ depends on $(a,b,p)$, but it is independent of $h$.
\end{lemma}
\begin{proof}[Proof of \Cref{lemma:lambda_bounded}]
We observe that for each $k\in [K]$ and $t\in I_k$ we have that
\begin{align*}
&\mathbb{E}\left[\|\partial_{x}((a - \zeta_a) \partial_x y)\|_{L^2(\mathcal{E})}^2\right]  = \mathbb{E}\left[\|\partial_{x}((a - \zeta_a^k) \partial_x y)\|_{L^2(\mathcal{E}^k)}^2+\|\partial_{x}( a \,\partial_x y)\|_{L^2((\mathcal{E}^k)^c)}^2\right]\\
&\hspace{0.5cm} \leq \mathbb{E}\left[\left\|\partial_{x}\left(\left(a - \left[\sum_{i\in \mathcal{B}_{\omega_k}}\frac{a\mathbf 1_{S^{i}_{\omega_k}}}{\pi_i}\right]\right) \partial_x y\right)\right\|_{L^2(\mathcal{E}^k)}^2\right]+\|\partial_{x}a \partial_x y + a\partial_{xx}y\|_{L^2((\mathcal{E}^k)^c)}^2.
\end{align*}
Here $S_{\omega_k}^i$ is set defined in \eqref{eq:def_Si}, evaluated in the random index $\omega_k$. Denote by $ \rho_k := 1 - \sum_{i\in \mathcal{B}_{\omega_k}}\mathbf 1_{S^{i}_{\omega_k}}/\pi_i$ the piece-wise function on space. Then, applying H\"older's inequality we have 
\begin{align*}
    &\mathbb{E}\left[\|\partial_{x}((a - \zeta_a) \partial_x y)\|_{L^2(\mathcal{E})}^2\right] \leq \mathbb{E}\left[\|\rho_k\partial_{x}a \partial_x y+\rho_ka\partial_{xx}y\|_{L^2(\mathcal{E}^k)}^2\right]\\&\hspace{6cm}+\|\partial_{x}a \partial_x y + a\partial_{xx}y\|_{L^2((\mathcal{E}^k)^c)}^2\\
    &\hspace{2cm}\leq2\mathbb{E}[|\rho_k|^2]\left(\|\partial_x a\|_{L^\infty(\mathcal{E}^k)}^2\|\partial_x y \|_{L^2(\mathcal{E}^k)}^2 +\|a\|_{L^\infty(\mathcal{E}^k)}^2\|\partial_{xx}y\|^2_{L^2(\mathcal{E}^k)} \right)\\
    &\hspace{2cm}\qquad+ 2\left(\|\partial_x a\|_{L^\infty(\mathcal{E})}^2\|\partial_x y\|_{L^2(\mathcal{E})}^2+ \| a\|_{L^\infty(\mathcal{E})}^2\|\partial_{xx} y\|_{L^2(\mathcal{E})}^2\right)\\
    &\hspace{2cm}\leq 2\max\{\mathbb{E}[|\rho_k|^2],1\}\left( \|a\|_{W^{1,\infty}(\mathcal{E})}^2\|y \|_{H^2(\mathcal{E})}^2\right).
\end{align*}
 for $t\in I_k$. On the other hand, let $\pi_\ast:=\min_i\pi_i>0$, then for each $e\in \mathcal{E}^k$ we have that $\mathbb{E}[|\rho_k|^2]\leq \frac{1}{\pi_\ast}-1 =:C_{\pi_\ast}$ for every $k$. Therefore, the same arguments lead us to
\begin{align*}
    \|\Lambda\|_{L^1(0,T)}\leq  2\max\{C_{\pi_\ast},1\}\left(\|a\|_{W^{1,\infty}(\mathcal{E})}^2+\|b\|_{L^{\infty}(\mathcal{E})}^2+\|p\|_{L^{\infty}(\mathcal{E})}^2\right)\|y\|^2_{L^2(0,T;H^2(\mathcal{E}))}.
\end{align*}
Since \Cref{th:well_posedness_nonhomo} holds, we conclude by applying \eqref{eq:energy_estimation_y_wellposedness}.
\end{proof}
Now, we continue with the convergence proof. 
\begin{proof}[Proof of  \Cref{th:convergence}]
    Let $y$ and $z$ denote the solutions of \eqref{eq:heat1} and \eqref{eq:heat1random}, respectively. Fix an arbitrary index $k \in [K]$. For every $t \in I_k$, we have
    \begin{align}\label{eq:ineq1}
      \langle \partial_t (y - z), y - z \rangle_{L^2(\mathcal{E})}
      &= \left\langle
           \partial_x(a\partial_x y)
           - \partial_x(\zeta_a\partial_x z),
           y - z
         \right\rangle_{L^2(\mathcal{E})} \nonumber \\
         &\hspace{-3cm}+
         \left\langle
           b\partial_x y
           - \zeta_b \partial_x z,
           y - z
         \right\rangle_{L^2(\mathcal{E})}+
         \left\langle
           py
           - \zeta_pz,
           y - z
         \right\rangle_{L^2(\mathcal{E})}
         +
         \left\langle
           (\zeta_1 - 1) f,
           y - z
         \right\rangle_{L^2(\mathcal{E})}
      \nonumber \\
      &=: P_1 + P_2 + P_3 + P_4.
    \end{align}

    In the following, we split the proof into three steps. First, we estimate $P_1$, and then,the terms $P_2, P_3,$ and $P_4$. Finally, we will conclude the proof by using the regularity of $z$ on time and applying Gr\"onwall's inequality. 
\smallbreak
    
    \noindent
    {\bf Step 1. $P_1$ estimation. } For simplicity, we denote by $\zeta_\psi = \zeta_\psi^k = \zeta^{\omega_k}_\psi$ for every $t\in I_k$. Now, for $P_1$, integrating by parts yields
    \begin{align*}
      P_1 
      &= 
      \big\langle 
        \partial_x(\zeta_a\partial_x(y - z)), 
        y - z
      \big\rangle_{L^2(\mathcal{E})}
      +
      \big\langle
        \partial_x\bigl((a - \zeta_a)\partial_x y\bigr),
        y - z
      \big\rangle_{L^2(\mathcal{E})}
      \nonumber \\
      &=
      \big\langle 
        \partial_x(\zeta_a\partial_x(y - z)), 
        y - z
      \big\rangle_{L^2(\mathcal{E})}
      +
      \big\langle
        \partial_x\bigl((a - \zeta_a)\partial_x y\bigr),
        y - z(t_{k-1})
      \big\rangle_{L^2(\mathcal{E})}
      \nonumber \\
      &\quad
      +
      \big\langle
        \partial_x\bigl((a - \zeta_a)\partial_x y\bigr),
        z(t_{k-1}) - z
      \big\rangle_{L^2(\mathcal{E})}
      \nonumber \\
      &=
      -\bigl\|\sqrt{\zeta_a}\partial_x(y - z)\bigr\|_{L^2(\mathcal{E}^k)}^2
      +
      \sum_{v \in \mathcal{V}}
      \sum_{e \in \mathcal{E}(v)}
         \zeta_a(v)\,\bigl(y_e - z_e\bigr)(v)\partial_x(y_e - z_e)(v)n_e(v)
      \nonumber \\
      &\quad
      +
      \big\langle
        \partial_x\bigl((a - \zeta_a)\partial_x y\bigr),
        y - z(t_{k-1})
      \big\rangle_{L^2(\mathcal{E})}
      +
      \big\langle
        \partial_x\bigl((a - \zeta_a)\partial_x y\bigr),
        z(t_{k-1}) - z
      \big\rangle_{L^2(\mathcal{E})}.
    \end{align*}

    \noindent
   Let us analyze the boundary term.  Since $\supp(\zeta_a) \subset \mathcal{G}^k$ on $[t_{k-1}, t_k]$, we have that 
\begin{align}\label{eq:support_boundary}
     \sum_{v \in \mathcal{V}}
      \sum_{e \in \mathcal{E}(v)}
         \Psi_e(v) =  \sum_{v \in \mathcal{V}^k}
      \sum_{e \in \mathcal{E}(v)}
         \Psi_e(v)
\end{align}
where $\Psi_e(v) = \zeta_a(v)\,\bigl(y_e - z_e\bigr)(v)\partial_x(y_e - z_e)(v)\,n_e(v)$, and  $\mathcal{V}^k$ is the set of vertices of $\mathcal{G}^k$. Observe that the set of vertices $\mathcal{V}^k$ can be decomposed as 
    \begin{align}\label{eq:decompose_boundary}
      \mathcal{V}^k
      =
      \bigl(\mathcal{V}_0^k\bigr)
      \cup
      \bigl(\mathcal{V}_b^k \,\cap\, \mathcal{V}_b\bigr)
         \,\cup\,
         \bigl(\mathcal{V}_b^k \,\cap\, \mathcal{V}_0\bigl),
    \end{align}
    see \Cref{fig:differnt_nodes}. Therefore, since \eqref{eq:support_boundary} and \eqref{eq:decompose_boundary}, we deduce that
\begin{align}\label{eq:decompostion_sum_boundary}
      \sum_{v \in \mathcal{V}^k}
      \sum_{e \in \mathcal{E}(v)}
         \Psi_e(v) =\sum_{v \in \mathcal{V}^k_0}
      \sum_{e \in \mathcal{E}(v)}
          \Psi_e(v)+\hspace{-0.4cm}\sum_{v \in \mathcal{V}_b^k \cap \mathcal{V}_b}
      \sum_{e \in \mathcal{E}(v)}\Psi_e(v)+\hspace{-0.4cm}\sum_{v \in \mathcal{V}_b^k \cap \mathcal{V}_0}
      \sum_{e \in \mathcal{E}(v)}
          \Psi_e(v).
    \end{align}
Now, by definition of $\zeta_a$, we have that
\begin{align}\label{eq:zeta_a_zero}
\zeta_a(v) = 0, \quad \text{for every}\quad v \in \bigl(\mathcal{V}_b^k \cap \mathcal{V}_0\bigr).
\end{align}
Moreover, by construction 
\begin{align}\label{eq:equal_boundary}
z_e(v,t) = y_e(v,t) = g_v(t), \quad \text{for every}\quad v \in \bigl(\mathcal{V}_b^k \cap \mathcal{V}_b\bigr).
\end{align}
Finally, it follows from \eqref{eq:decompostion_sum_boundary}, \eqref{eq:zeta_a_zero} and \eqref{eq:equal_boundary} that
\begin{align*}
\sum_{v \in \mathcal{V}}\sum_{e \in \mathcal{E}(v)}
   \Psi_e(v) &=\sum_{v \in \mathcal{V}^k_0}
   (y - z)(v)
   \bigg(
     \sum_{e \in \mathcal{E}(v)} \zeta_a(v)\partial_x y_e(v)n_e(v)-\zeta_a(v)\partial_x z_e(v)n_e(v)
   \bigg)\\
   &= \sum_{v \in \mathcal{V}^k_0}
   (y - z)(v)
     \sum_{e \in \mathcal{E}(v)} \zeta_a(v)\partial_x y_e(v)n_e(v)
   \\
   &=\sum_{v \in \mathcal{V}^k_0}
   (y - z)(v)
     \sum_{e \in \mathcal{E}(v)} \sum_{i\in \mathcal{B}_{\omega_k}}\frac{1_{S_{\omega_k}^i}(v)a_e(v)}{\pi_i}\partial_x y_e(v)n_e(v)
     \\
   &=\sum_{v \in \mathcal{V}^k_0}
   (y - z)(v)
      \sum_{i\in \mathcal{B}_{\omega_k}}\frac{1_{S_{\omega_k}^i}(v)}{\pi_i}\sum_{e \in \mathcal{E}(v)}a_e(v)\partial_x y_e(v)n_e(v) = 0,
\end{align*}
where we have used that $y,z \in H^1(\mathcal{E}^k)$ and both satisfy Kirchhoff's condition on $\mathcal{V}_{0}^k$. 
Consequently, $P_1$ becomes
\begin{align}\label{eq:last_relation_p1}
P_1
&=
-\Bigl\|\sqrt{\zeta_a}\partial_x (y - z)\Bigr\|_{L^2(\mathcal{E}^k)}^2
+
\bigl\langle
  \partial_x\bigl((a - \zeta_a)\partial_x y\bigr),
  y - z(t_{k-1})
\bigr\rangle_{L^2(\mathcal{E})}
\nonumber \\
&\hspace{4.2cm}
+
\bigl\langle
  \partial_x\bigl((a - \zeta_a)\partial_x y\bigr),
  z(t_{k-1}) - z
\bigr\rangle_{L^2(\mathcal{E})}.
\end{align}

\noindent
{\bf Step 2. $P_2$, $P_3$ and $P_4$ estimation. } Since for every $t\in (t_{k-1},t_k)$ we have $\supp(\zeta_b)\subset \mathcal{E}^k $, for $P_2$ we obtain that
\begin{align}\label{eq_inep2}
P_2
&=
\bigl\langle
  \zeta_b\partial_x (y - z),
  y - z
\bigr\rangle_{L^2(\mathcal{E})}
+
\bigl\langle
  (b - \zeta_b)\partial_x y,
  y - z
\bigr\rangle_{L^2(\mathcal{E})}
\nonumber \\
&=
\bigl\langle
  \zeta_b\partial_x (y - z),
  y - z
\bigr\rangle_{L^2(\mathcal{E}^k)}
+
\bigl\langle
  (b - \zeta_b)\partial_x y,
  y - z(t_{k-1})
\bigr\rangle_{L^2(\mathcal{E})}
\nonumber \\
&\hspace{4.4cm}
+
\bigl\langle
  (b - \zeta_b)\partial_x y,
  z(t_{k-1}) - z
\bigr\rangle_{L^2(\mathcal{E})}.
\end{align}
Since $a$ satisfies \eqref{eq:assumption_coef_a}, $ \zeta_a^k>0$ in $\mathcal{E}^k$. Thus, using  Hölder’s inequality in \eqref{eq_inep2} yields
\begin{align}\label{eq:ineqp22}
\nonumber P_2
\leq
\Bigl\|\sqrt{\zeta_a}\partial_x (y - z)\Bigr\|_{L^2(\mathcal{E}^k)}
\Bigl\|\tfrac{\zeta_b}{\sqrt{\zeta_a}}(y - z)\Bigr\|_{L^2(\mathcal{E}^k)}
&+
\bigl\langle
  (b - \zeta_b)\,\partial_x y,
  y - z(t_{k-1})
\bigr\rangle_{L^2(\mathcal{E})}
 \\
&
+
\bigl\langle
  (b - \zeta_b)\partial_x y,
  z(t_{k-1}) - z
\bigr\rangle_{L^2(\mathcal{E})}.
\end{align}
Now, applying Young’s inequality to \eqref{eq:ineqp22} we get
\begin{align}\label{eq:last_P2}
P_2
&\le
\tfrac12
\Bigl\|\sqrt{\zeta_a}\partial_x (y - z)\Bigr\|_{L^2(\mathcal{E}^k)}^2
+
\tfrac12
\Bigl\|\tfrac{\zeta_b}{\sqrt{\zeta_a}}\Bigr\|_{L^\infty(\mathcal{E}^k)}^2
\Bigl\|y - z\Bigr\|_{L^2(\mathcal{E})}^2
\nonumber \\
&\quad
+
\bigl\langle
  (b - \zeta_b)\partial_x y,
  y - z(t_{k-1})
\bigr\rangle_{L^2(\mathcal{E})}
+
\bigl\langle
  (b - \zeta_b)\partial_x y,
  z(t_{k-1}) - z
\bigr\rangle_{L^2(\mathcal{E})}.
\end{align}
Finally, we estimate $P_3 + P_4$; by applying Hölder’s inequality, we get
\begin{align}\label{eq:last_P34}
P_3 + P_4
&=
\bigl\langle
  \zeta_p(y - z),
  \,y - z
\bigr\rangle_{L^2(\mathcal{E})}
+
\bigl\langle
  (p - \zeta_p)y,
  y - z
\bigr\rangle_{L^2(\mathcal{E})}
+
\bigl\langle
  (\zeta_1 - 1)f,
  y - z
\bigr\rangle_{L^2(\mathcal{E})}
\nonumber \\
&\le
\|\zeta_p\|_{L^\infty(\mathcal{E}^k)}\|y - z\|_{L^2(\mathcal{E}^k)}^2
+
\bigl\langle
  (p - \zeta_p)y,
  y - z(t_{k-1})
\bigr\rangle_{L^2(\mathcal{E})}
\nonumber \\
&\quad
+
\bigl\langle
  (p - \zeta_p)y,
  \,z(t_{k-1}) - z
\bigr\rangle_{L^2(\mathcal{E})}
+
\bigl\langle
  (\zeta_1 - 1)f,
  y - z(t_{k-1})
\bigr\rangle_{L^2(\mathcal{E})}
\nonumber \\
&\quad
+
\bigl\langle
  (\zeta_1 - 1)f,
  z(t_{k-1}) - z
\bigr\rangle_{L^2(\mathcal{E})}.
\end{align}
Before proceeding, note that $z(t_{k-1})$ depends on the realization of the random variable $\omega_{k-1}$. By the independence of the sequence $\{\omega_k\}_{k=1}^K$, the value $z(t_{k-1})$ is independent of $\zeta_\psi^k$, which is defined through $\omega_k$. Thus, for any functions $\psi, \rho, \phi: \mathcal{G} \to \mathbb{R}$, we have
\begin{align}\label{eq:relation_Expectation}
\nonumber
\mathbb{E}\Bigl[\langle (\zeta_\psi - \psi)\,\phi,\,\rho - z(t_{k-1})\rangle_{L^2(\mathcal{E})}\Bigr]
&=
\Bigl\langle
  \mathbb{E}\bigl[(\zeta_\psi - \psi)\,\phi\bigr],
  \,\mathbb{E}\bigl[\rho - z(t_{k-1})\bigr]
\Bigr\rangle_{L^2(\mathcal{E})}
\\
&=
\Bigl\langle
  \mathbb{E}\bigl[\zeta_\psi - \psi\bigr]\,\phi,
  \,\mathbb{E}\bigl[\rho - z(t_{k-1})\bigr]
\Bigr\rangle_{L^2(\mathcal{E})}
= 0.
\end{align}
Now, taking expectation in \eqref{eq:ineq1} and substituting \eqref{eq:last_relation_p1}, \eqref{eq_inep2}, \eqref{eq:last_P2}, and \eqref{eq:last_P34}, some terms vanish due to \eqref{eq:relation_Expectation}, we obtain
\begin{align}\label{eq:fist_after_expectation}
\nonumber\frac{1}{2}\,\partial_t \Bigl(\mathbb{E}\bigl[\|\,y - z\|_{L^2(\mathcal{E})}^2\bigr]\Bigr)&
+
\frac{1}{2}
\mathbb{E}\Bigl[\bigl\|\sqrt{\zeta_a}\partial_x(y - z)\bigr\|_{L^2(\mathcal{E}^k)}^2\Bigr]
\\
\nonumber&\hspace{-1.9cm}\le
\mathbb{E}\bigl[\langle \partial_x\bigl((a - \zeta_a)\partial_x y\bigr),z(t_{k-1}) - z\rangle_{L^2(\mathcal{E})}\bigr]
+
\mathbb{E}\bigl[\bigl(C_{\zeta_a,\zeta_b} + C_p\bigr)\|y - z\|_{L^2(\mathcal{E})}^2\bigr]
\\
\nonumber&\hspace{-1.9cm}\quad
+
\mathbb{E}\bigl[\langle (b - \zeta_b)\partial_x y,\,z(t_{k-1}) - z\rangle_{L^2(\mathcal{E})}\bigr]
+
\mathbb{E}\bigl[\langle (p - \zeta_p)y,z(t_{k-1}) - z\rangle_{L^2(\mathcal{E})}\bigr]
\\
&\hspace{-1.9cm}\quad
+
\mathbb{E}\bigl[\langle (\zeta_1 - 1)f,z(t_{k-1}) - z\rangle_{L^2(\mathcal{E})}\bigr],
\end{align}
where $C_{\zeta_a,\zeta_b} =\frac{1}{2} \bigl\|\tfrac{\zeta_b}{\sqrt{\zeta_a}}\bigr\|_{L^\infty(\mathcal{E}^k)}^2$ and $C_p = \|p\|_{L^\infty(\mathcal{E})}^2/\pi_\ast$ with $\pi_\ast:=\min_i\pi_i>0$. Then, due to the Cauchy-Schwarz inequality and the definition of $\zeta_\psi$, we have that
\begin{align*}
C_{\zeta_a,\zeta_b} &= \left\| \frac{\displaystyle\sum_{i\in \mathcal{B}_{\omega_k}}{ \frac{\chi^{i,\omega_k}_b}{\pi_i}}}{\left(\displaystyle\sum_{i\in \mathcal{B}_{\omega_k}}{ \frac{\chi^{i,\omega_k}_a}{\pi_i}}\right)^{\frac{1}{2}}}\right\|_{L^\infty(\mathcal{E}^k)}^2\hspace{-0.3cm}\leq \hspace{0.1cm} \left\| \frac{\left(\displaystyle\sum_{i\in \mathcal{B}_{\omega_k}}{ \frac{(\chi^{i,\omega_k}_b)^2}{\chi^{i,\omega_k}_a\pi_i}}\right)^{\frac{1}{2}}\left(\displaystyle\sum_{i\in \mathcal{B}_{\omega_k}}{ \frac{\chi^{i,\omega_k}_a}{\pi_i}}\right)^{\frac{1}{2}}}{\left(\displaystyle\sum_{i\in \mathcal{B}_{\omega_k}}{ \frac{\chi^{i,\omega_k}_a}{\pi_i}}\right)^{\frac{1}{2}}}\right\|_{L^\infty(\mathcal{E}^k)}^2\\
&\leq \frac{1}{\pi_\ast}\left\|\frac{b^2}{a} \right\|_{L^\infty(\mathcal{E}^k)}\leq \frac{\|b\|_{L^\infty(\mathcal{E})}^2}{\pi_\ast a_0}=: C_{a_0,b},
\end{align*}
being $a_0>0$ the ellipticity constant from \eqref{eq:assumption_coef_a}. Now, using Hölder’s inequality in \eqref{eq:fist_after_expectation} for the $L^2(\mathcal{E})$ product and for the expectation, we get
\begin{align}\label{eq:simplify_relation}
\partial_t \Bigl(\mathbb{E}\bigl[\|y - z\|_{L^2(\mathcal{E})}^2\bigr]\Bigr)
\le
C\mathbb{E}\bigl[\|y - z\|_{L^2(\mathcal{E})}^2\bigr]
+
4\Lambda^{\frac{1}{2}}(t)\mathbb{E}\bigl[\|z(t_{k-1}) - z\|_{L^2(\mathcal{E})}^2\bigr]^{\frac{1}{2}},
\end{align}
where $C = 2(C_{a_0,b} + C_p)$, and $\Lambda$ is defined in \eqref{eq:def_lambda}. 
\smallbreak
\noindent
{\bf Step 3. Error composition.}  Integrating \eqref{eq:simplify_relation} over $(t_{k-1}, t)$ gives
\begin{align}\label{eq:ine_from_k}
\nonumber
\mathbb{E}\bigl[\|y(t) - z(t)\|_{L^2(\mathcal{E})}^2\bigr]&
\le
\mathbb{E}\bigl[\|y(t_{k-1}) - z(t_{k-1})\|_{L^2(\mathcal{E})}^2\bigr]\\
\nonumber
&\hspace{-3cm}+
C \int_{t_{k-1}}^{t}
\mathbb{E}\bigl[\|y(s) - z(s)\|_{L^2(\mathcal{E})}^2\bigr]\,
ds
+ 4 \int_{t_{k-1}}^{t}
\Lambda^{\frac{1}{2}}(s)\,\mathbb{E}\Bigl[\|z(t_{k-1}) - z(s)\|_{L^2(\mathcal{E})}^2\Bigr]^{\frac{1}{2}}ds
\\
\nonumber&\hspace{-2cm}\le
\mathbb{E}\bigl[\|y(t_{k-1}) - z(t_{k-1})\|_{L^2(\mathcal{E})}^2\bigr]
+
C \int_{t_{k-1}}^{t}
\mathbb{E}\bigl[\|y(s) - z(s)\|_{L^2(\mathcal{E})}^2\bigr]
ds
\\
&\hspace{-1.0cm}\quad
+ 4 \Bigl(\int_{t_{k-1}}^{t}\Lambda(s)ds\Bigr)^{\frac{1}{2}}
\Bigl(\int_{t_{k-1}}^{t}
\mathbb{E}\bigl[\|z(t_{k-1}) - z(s)\|_{L^2(\mathcal{E})}^2\bigr]ds\Bigr)^{\frac{1}{2}}.
\end{align}
Now, composing the estimate \eqref{eq:ine_from_k} for every $k\in [k_t]$ yields
\begin{align*}
&\mathbb{E}\bigl[\|y(t) - z(t)\|_{L^2(\mathcal{E})}^2\bigr]
\le
\mathbb{E}\bigl[\|y(0) - z(0)\|_{L^2(\mathcal{E})}^2\bigr]
+
C \int_{0}^{t}
\mathbb{E}\bigl[\|y(s) - z(s)\|_{L^2(\mathcal{E})}^2\bigr]
ds
\\
&\hspace{2.5cm}\quad
+4 \sum_{k=1}^{k_t - 1}
\Bigl(\int_{t_{k-1}}^{t_{k}}\Lambda(s)\,ds\Bigr)^{\frac{1}{2}}
\Bigl(\int_{t_{k-1}}^{t_{k}}
\mathbb{E}\bigl[\|z(t_{k-1}) - z(s)\|_{L^2(\mathcal{E})}^2\bigr]
ds\Bigr)^{\!\frac{1}{2}}
\\
&\hspace{2.5cm}\quad
+4
\Bigl(\int_{t_{k_t - 1}}^{t}\Lambda(s)\,ds\Bigr)^{\frac{1}{2}}
\Bigl(\int_{t_{k_t - 1}}^{t}
\mathbb{E}\bigl[\|z(t_{k_t - 1}) - z(s)\|_{L^2(\mathcal{E})}^2\bigr]
ds\Bigr)^{\!\frac{1}{2}}.
\end{align*}
Since $y(0) = z(0)$, applying the integral Gr\"onwall lemma (see \cite[Appendix B.k]{evans_book}) gives
\begin{align*}
&\mathbb{E}\bigl[\|y(t) - z(t)\|_{L^2(\mathcal{E})}^2\bigr]\le 4e^{Ct}\Bigl(\int_{t_{k_t - 1}}^{t}\Lambda(s)ds\Bigr)^{\frac{1}{2}}
   \Bigl(\int_{t_{k_t - 1}}^{t}
     \mathbb{E}\bigl[\|z(t_{k_t - 1}) - z(s)\|_{L^2(\mathcal{E})}^2\bigr]
     ds\Bigr)^{\frac{1}{2}}\\
     &\hspace{2cm}+
4e^{Ct}\Biggl[
  \sum_{k=1}^{k_t - 1}
    \Bigl(\int_{t_{k-1}}^{t_{k}}\Lambda(s)ds\Bigr)^{\frac{1}{2}}
    \Bigl(\int_{t_{k-1}}^{t_{k}}
      \mathbb{E}\bigl[\|z(t_{k-1}) - z(s)\|_{L^2(\mathcal{E})}^2\bigr]
      ds\Bigr)^{\frac{1}{2}}
\Biggr].
\end{align*}
Applying the Cauchy-Schwarz inequality to the sum yields
\begin{align}\label{eq:after_gronwall}
\nonumber
\mathbb{E}\bigl[\|\,y(t) - z(t)\|_{L^2(\mathcal{E})}^2\bigr]
&\le
4e^{Ct}
\Biggl[
  \Bigl(\int_{0}^{t}\!\Lambda(s)ds\Bigr)^{\frac{1}{2}}
  \Biggl(
    \sum_{k=1}^{k_t - 1}
      \int_{t_{k-1}}^{t_{k}}
         \mathbb{E}\bigl[\|z(t_{k-1}) - z(s)\|_{L^2(\mathcal{E})}^2\bigr]\,
      ds
\\
&\quad
\left.
    +
    \int_{t_{k_t - 1}}^{t}
      \mathbb{E}\bigl[\|\,z(t_{k_t - 1}) - z(s)\|_{L^2(\mathcal{E})}^2\bigr]
    ds
  \Biggr)^{\frac{1}{2}}
\right].
\end{align}
On the other hand, since $z \in H^1(0,T; L^2(\mathcal{E}))$, we have
\begin{align}\label{eq:abs_z}
\nonumber\mathbb{E}\bigl[\|z(t_{k-1}) - z(s)\|_{L^2(\mathcal{E})}^2\bigr]
&=
\mathbb{E}\Bigl[\Bigl\|\int_{t_{k-1}}^s \! \! \partial_t z(\tau)d\tau\Bigr\|_{L^2(\mathcal{E})}^2\Bigr]\le
\mathbb{E}\Bigl[\Bigl(\int_{t_{k-1}}^s \! \!\|\partial_t z(\tau)\|_{L^2(\mathcal{E})}\,d\tau\Bigr)^2\Bigr]\\
&\le
\mathbb{E}\Bigl[h\int_{t_{k-1}}^{t_k}  \! \!\|\partial_t z(\tau)\|_{L^2(\mathcal{E})}^2d\tau \Bigr]=
h\int_{t_{k-1}}^{t_k}
\mathbb{E}\bigl[\|\partial_t z(\tau)\|_{L^2(\mathcal{E})}^2\bigr]d\tau.
\end{align}
Substituting \eqref{eq:abs_z} into \eqref{eq:after_gronwall} we get
\begin{align*}
\mathbb{E}\bigl[\|y(t) - z(t)\|_{L^2(\mathcal{E})}^2\bigr]&\leq4e^{Ct}\Biggl[\Bigl(\int_{0}^{t} \Lambda(s)ds\Bigr)^{\frac{1}{2}}
\Biggl(h \sum_{k=1}^{k_t - 1}\int_{t_{k-1}}^{t_k}\int_{t_{k-1}}^{t_k}\mathbb{E}\bigl[\|\partial_t z(\tau)\|_{L^2(\mathcal{E})}^2\bigr]d\tau ds\\
&\hspace{2cm}\quad+h\int_{t_{k_t - 1}}^{t}\int_{t_{k-1}}^{t_k}
\mathbb{E}\bigl[\|\partial_t z(\tau)\|_{L^2(\mathcal{E})}^2\bigr]
d\tau ds\Biggr)^{\frac{1}{2}}\Biggr]\\
&\le 4e^{CT}h\Bigl(\int_{0}^{t}\Lambda(s)ds\Bigr)^{\frac{1}{2}}
\Bigl(\int_{0}^{t}\mathbb{E}\bigl[\|\partial_t z(\tau)\|_{L^2(\mathcal{E})}^2\bigr]d\tau\Bigr)^{\frac{1}{2}}. 
\end{align*}
Now, applying the energy estimate from Theorem \ref{th:wellpossednes_random}, we obtain 
\begin{align*}
\mathbb{E}\bigl[\|y(t) - z(t)\|_{L^2(\mathcal{E})}^2\bigr]  &\le4he^{CT}\|\Lambda\|_{L^1(0,T)}^{\frac{1}{2}}\mathbb{E}\bigl[\|z\|_{H^1(0,T;L^2(\mathcal{E}))}^2\bigr]^{\frac{1}{2}}\leq4he^{CT}\hat{C}\|\Lambda\|_{L^1(0,T)}^{\frac{1}{2}},
\end{align*}
with $\hat{C} = C_T(\|g\|_{H^1(0,T;l^2(\mathcal{V}_b))} + \|f\|_{L^2(0,T;L^2(\mathcal{E}))} + \|y^0\|_{H^1(\mathcal{E})}).$ Completing the proof.
\end{proof}

\section{Numerical simulations}\label{sec:numerical_simulations}
In this section, numerical examples are presented to validate \Cref{th:convergence}. Moreover, we compare the memory usage and computational time required to solve \eqref{eq:heat1random} with those needed to solve \eqref{eq:heat1} on the full graph. We consider the graph depicted in \Cref{fig_1_deamon_graph}. Two examples are considered: the first uses the batches from \textit{option 2} in \Cref{sec:example_decomp}, and the second uses those from \textit{option 1}.

\subsection{Example 1}\label{section:example1}
 We begin by introducing the finite element framework for systems \eqref{eq:heat1} and \eqref{eq:heat1random}.

\noindent
\textbf{Finite Element Setting.} For the discretization of the parabolic equation \eqref{eq:heat1} on the graph, we employ a finite element discretization using standard piecewise linear ($P1$) basis functions (see \cite{firstpaper} for an illustrative introduction). Let $n\geq 1$, and denote by $\mathfrak{V}$ the finite element space spanned by the $P1$ basis functions $\{\phi_i\}_{i=1}^n$. The semi-discrete formulation of \eqref{eq:heat1} consists of finding $y_n\in C^1(0,T;\mathfrak{V})$ that satisfies
\begin{align}\label{eq:semi_discrete_parabolic}
    \begin{cases}
        \left(\partial_t y_n,\phi_i\right)_{L^2(\mathcal{E})} + \left(a\,\partial_x y_n, \partial_x \phi_i\right)_{L^2(\mathcal{E})} + \left(b\,\partial_x y_n, \phi_i\right)_{L^2(\mathcal{E})}\\\hspace{6.5cm}+ \left(p\,y_n, \phi_i\right)_{L^2(\mathcal{E})} = \left(f, \phi_i\right)_{L^2(\mathcal{E})},\\[1mm]
        (y_n(0,\cdot), \phi_i)_{L^2(\mathcal{E})} = (y^0, \phi_i)_{L^2(\mathcal{E})},
    \end{cases}
\end{align}
for every $t\in (0,T)$ and $i\in[n]$, with the condition $y_n(t,v)=g(t)$ for all $v\in \mathcal{V}_b$.

Next, consider the decomposition illustrated in \Cref{fig_1_deamon_graph}. In this decomposition, the graph is partitioned into four subgraphs $\{\mathcal{G}_i\}_{i=1}^4$. To introduce the RBM framework, we consider the batches from \textit{option 2} in \Cref{sec:example_decomp}:
\begin{align}\label{eq:decompostion1}
    \mathcal{B}_1 =\{1\},\quad \mathcal{B}_2 =\{2\},\quad \mathcal{B}_3 =\{3\},\quad \mathcal{B}_4 =\{4\},\quad \mathcal{B}_5 =[4].
\end{align}
As is observed in \Cref{sec:example_decomp}, this decomposition and the associated batches satisfy \Cref{assump:A1}. This batch selection imposes that either one subgraph is chosen or the full graph is chosen. For $K\in \N$, we introduce the random variables $\{\omega_k\}_{k=1}^K$, with associated probabilities $p_j=1/5$ for each $j\in[5]$. Note that $\pi_i=2/5$ for $i\in[4]$. 

The implementation of \eqref{eq:heat1random} is carried out using finite elements together with \Cref{alg1}. For each $k\in[K]$, on the randomly selected graph $\mathcal{E}^k$, the semi-discrete problem reads as:
\begin{align}\label{eq:semi_discrete_random}
    \begin{cases}
        \left(\partial_t \eta_n,\phi_i\right)_{L^2(\mathcal{E}^k)} + \left(\zeta_a\,\partial_x \eta_n, \partial_x \phi_i\right)_{L^2(\mathcal{E}^k)} + \left(\zeta_b\,\partial_x \eta_n, \phi_i\right)_{L^2(\mathcal{E}^k)}\\
        \hspace{5.5cm}+ \left(\zeta_p\,\eta_n, \phi_i\right)_{L^2(\mathcal{E}^k)} = \left(\zeta_1 f, \phi_i\right)_{L^2(\mathcal{E}^k)},\\[1mm]
        (\eta_n(\cdot,t_{k-1}), \phi_i)_{L^2(\mathcal{E}^k)} = (z_n(\cdot,t_{k-1}), \phi_i)_{L^2(\mathcal{E}^k)},
    \end{cases}
\end{align}
for every $t\in I_k$ and $i\in[n]$. The function $\eta_n$ is subject to the boundary conditions $\eta_n(v,t)=\xi_n(v,t)$ for every $v\in \mathcal{V}_0\cap \mathcal{V}_b^k$ and $\eta_n(v,t)=g_v(t)$ for every $v\in \mathcal{V}_b\cap \mathcal{V}_b^k$, with $\xi_n(x,t)=z_n(x,t_{k-1})$ for all $(x,t)\in e\times I_k$ for every $e\in (\mathcal{E}^k)^c$, and $z_n(x,0) =y^0(x)$ for all $x\in e$ for every $e\in \mathcal{E}$.

From now on, each edge $e\in\mathcal{E}$ is identified with the interval $(0,1)$ (after fixing an orientation). Also, the PDE coefficients will be taken as 
\begin{align*}
    a_e(x) = x(x-1)+\frac{1}{2},\quad b_e(x) = \frac{\sin(\pi x)}{2},\quad p_e(x)= \sin(\pi x),\quad\text{for every }e\in\mathcal{E}.
\end{align*}
To construct an explicit solution, the method of manufactured solutions (see \cite[Chapter 12]{MR3931345}) is used. Specifically, we prescribe a function $y$ on $\mathcal{G}\times(0,T)$ that satisfies the continuity and Kirchhoff conditions. Then, the corresponding source term $f$, boundary data $g$, and initial condition $y_0$ are computed as
\begin{align*}
    f_e:= \partial_t y_e
- \partial_x\bigl(a_e\,\partial_x y_e \bigr)
+ b_e\,\partial_x y_e
+ p_e\,y_e, \quad g_v(t) := y_e(v,t), \quad y_e^0(x)=y_e(x,0),
\end{align*}
for every $e\in \mathcal{E}$ and $v\in\mathcal{V}_b$. We assume that
\begin{align*}
    y_e(x,t) = w_e(x)v(t), \hspace{0.3cm} \text{with} \hspace{0.3cm} w_e(x)=\alpha_e x^4+\beta_e x^3+ x^2+ x+1,
\end{align*}
and $ v(t)= \sin(2\pi t)$, where $\alpha_e$ and $\beta_e$ are given in \Cref{tab:coefficients}. These coefficients are chosen such that the exact solution satisfies the continuity and Kirchhoff conditions.

\begin{table}[h!]
\centering
\caption{Coefficients $\alpha_e$ and $\beta_e$ for each edge.}\label{tab:coefficients}
\begin{tabular}{c|cccccccccc}
\toprule
        Coefficients      & $e_1$ & $e_2$ & $e_3$ & $e_4$ & $e_5$ & $e_6$ & $e_7$ & $e_8$ & $e_9$ & $e_{10}$ \\
\midrule
$\alpha_e$   & 10      & -3      & 5      & 4      & 4      & 10      & -3      & 5      & -3      & 4       \\
$\beta_e$    & -12     & 1       & -7     & -6     & -6     & -12     & 1       & -7     & 1       & -6      \\
\bottomrule
\end{tabular}
\end{table}

\noindent
\textbf{Numerical Implementation.}
For the time discretization of \eqref{eq:semi_discrete_parabolic} and \eqref{eq:semi_discrete_random}, we compare the first-order methods implicit Euler (IE), the $\theta$-method, and the semi-implicit Euler method (SIEM), and the second-order method Crank-Nicolson (CN).

For each edge $e\in\mathcal{E}$, $100$ basis functions are introduced per edge; therefore, we use $1000$ for all edges. Moreover, to include the vertex (interior and boundary) conditions, we consider a basis function for each vertex. Since the graph in \eqref{fig_1_deamon_graph} has $10$ vertices, the finite element space $\mathfrak{V}$ is spanned by $n=1010$ basis functions. We set $T=1$ and use a time step of $\Delta t = 0.002$. Using the previous schemes, the computational results obtained by solving \eqref{eq:semi_discrete_parabolic} are summarized in \Cref{tab:deterministic}. Since the CN scheme is second-order, its error is smaller than that of the other methods, albeit at the expense of higher computational cost. In \Cref{tab:deterministic}, we define
\begin{align*}
    \text{Error}:=\sup_{t\in[0,T]}\|y(t)-y_{n}(t)\|^2_{L^2(\mathcal{E})}.
\end{align*}
\begin{table}
\centering
\caption{Deterministic solver comparison on the full graph problem}\label{tab:deterministic}
\resizebox{\textwidth}{!}{
\begin{tabular}{l|cccc}
\toprule
\textbf{Metric}        & \textbf{IE}        & \textbf{CN}         & \boldmath$\theta=\tfrac34$ & \textbf{SIEM}    \\
\midrule
Error                  & $3.764\times10^{-1}$ & $7.327\times10^{-2}$  & $3.747\times10^{-1}$   & $3.689\times10^{-1}$ \\
Execution Time (s)              & $57.13$             & $74.68$              & $59.79$                & $56.71$            \\
Memory Usage (MB)            & $200.14$            & $208.00$             & $200.36$               & $199.93$           \\
\bottomrule
\end{tabular}}
\end{table}

\begin{figure}
    \centering
    \includegraphics[width=1\linewidth]{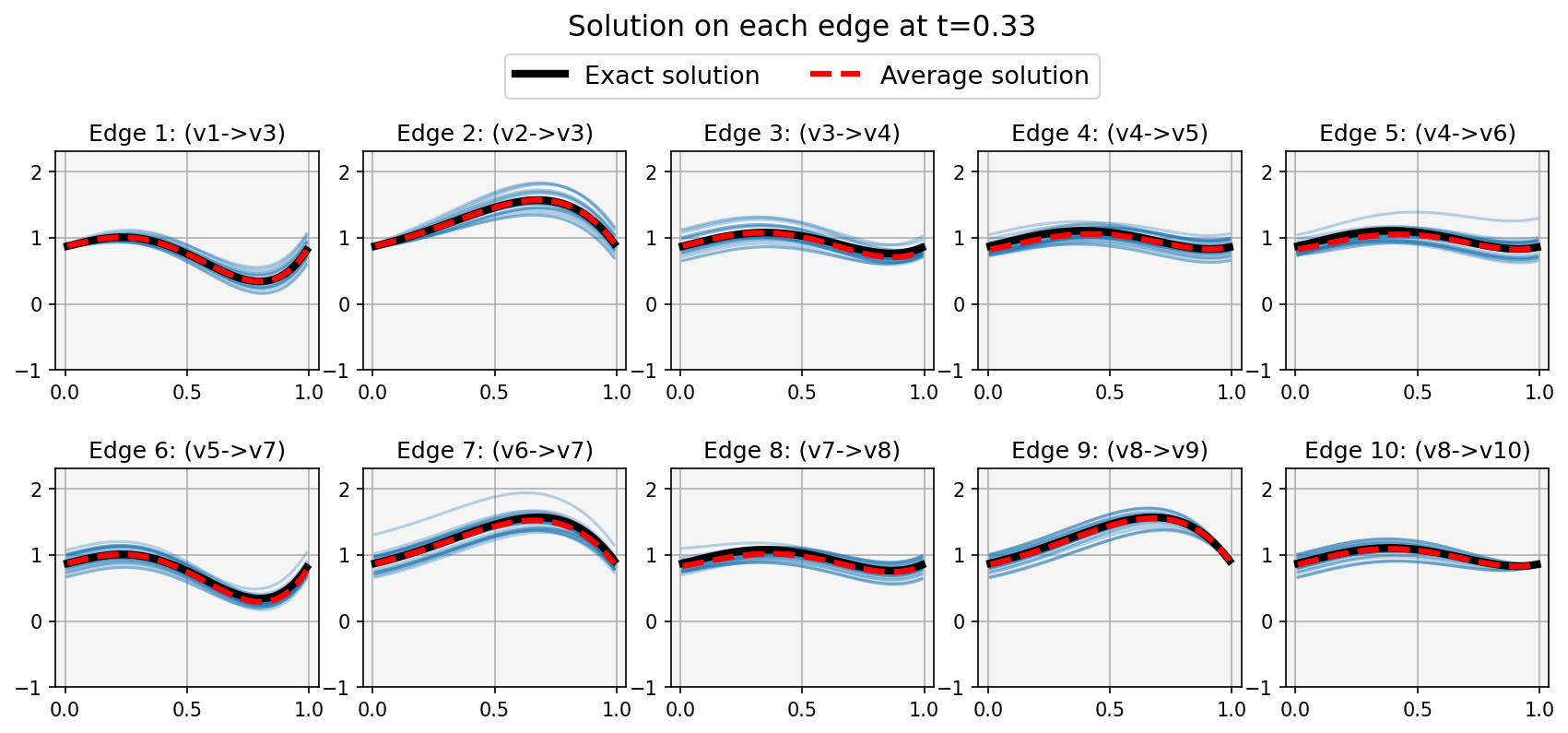}
    \caption{Illustration of the exact solution and the average solution on each edge. The exact solution corresponds to the black continuous lines, while the red dashed lines are the average solution. Thin blue lines illustrate the different realizations.}
    \label{fig:trajectories}
\end{figure}

For the RBM, equation \eqref{eq:heat1random} is solved using $100$ basis functions per edge on the randomly selected graph, with a time step $\Delta t = 0.002$. For $h\in\{0.002,0.004,0.006\}$, the results obtained by using the same time-schemes are shown in \Cref{table:rbm_dec1,table:rbm_dec1_theta_siem}. To approximate the expected value, $20$ realizations are employed. The results for the case $h=0.002$ are shown in \Cref{fig:trajectories}. In \Cref{table:rbm_dec1}, we define
\begin{align}\label{eq:definition_error}
    \text{Error}_1 = \sup_{t\in [0,T]} \mathbb{E}\left[\|y(t)-z(t)\|^2_{L^2(\mathcal{E})}\right],\,\,\, \text{Error}_2 = \sup_{t\in [0,T]} \|y(t)-\mathbb{E}\left[z(t)\right]\|^2_{L^2(\mathcal{E})}.
\end{align}
Further, in \Cref{table:rbm_dec1}, the average execution time and memory usage refer to the mean computational time and memory consumption, respectively. Note that the average execution time is approximately half that required to solve the equation on the full graph. Although RBM uses less memory, the difference is not substantial because, while the equation is solved on a small subgraph when one of the batches $\mathcal{B}_1$, $\mathcal{B}_2$, $\mathcal{B}_3$, or $\mathcal{B}_4$ is selected, the full graph is used when $\mathcal{B}_5$ is chosen. Furthermore, the best approximation is achieved for $h=0.002$, in accordance with \Cref{th:convergence}. Note that employing the CN scheme does not enhance accuracy in this context, which suggests that RBM converges more slowly than a second-order method. This observation does not contradict \Cref{th:convergence}, which guarantees an order of at least one.
\begin{table}[htbp]
\centering
\caption{Performance of RBM with IE and CN, using decomposition \eqref{eq:decompostion1}.}\label{table:rbm_dec1}
\resizebox{\textwidth}{!}{
\begin{tabular}{l|cc|cc|cc}
\toprule
\multirow{2}{*}{\textbf{Metric}} 
    & \multicolumn{2}{c|}{$h=0.002$} 
    & \multicolumn{2}{c|}{$h=0.004$} 
    & \multicolumn{2}{c}{$h=0.006$} \\
\cmidrule(r){2-3} \cmidrule(r){4-5} \cmidrule(r){6-7}
    & \textbf{IE} & \textbf{CN} & \textbf{IE} & \textbf{CN} & \textbf{IE} & \textbf{CN} \\
\midrule
$\mathrm{Error}_1$             & $3.896\times10^{-1}$ & $3.866\times10^{-1}$ & $7.904\times10^{-1}$ & $7.910\times10^{-1}$ & $1.105\times10^{0}$ & $1.109\times10^{0}$ \\
$\mathrm{Error}_2$             & $4.273\times10^{-2}$ & $3.576\times10^{-2}$ & $2.092\times10^{-1}$ & $2.238\times10^{-1}$ & $1.474\times10^{-1}$ & $1.409\times10^{-1}$ \\
Av.\ Execution Time (s)        & $27.153$             & $38.144$             & $22.278$             & $32.382$             & $28.506$             & $40.005$             \\
Av.\ Memory Usage (MB)         & $163.945$            & $165.754$            & $182.832$            & $183.754$            & $183.441$            & $179.992$            \\
\bottomrule
\end{tabular}}
\end{table}

\begin{table}[htbp]
\centering
\caption{Performance RBM with $\theta$-method and SIEM, using decomposition \eqref{eq:decompostion1}.}\label{table:rbm_dec1_theta_siem}
\resizebox{\textwidth}{!}{
\begin{tabular}{l|cc|cc|cc}
\toprule
\multirow{2}{*}{\textbf{Metric}} 
    & \multicolumn{2}{c|}{$h=0.002$} 
    & \multicolumn{2}{c|}{$h=0.004$} 
    & \multicolumn{2}{c}{$h=0.006$} \\
\cmidrule(r){2-3} \cmidrule(r){4-5} \cmidrule(r){6-7}
    & $\boldsymbol\theta$ & \textbf{SIEM} & $\boldsymbol\theta$ & \textbf{SIEM} & $\boldsymbol\theta$ & \textbf{SIEM} \\
\midrule
$\mathrm{Error}_1$             & $3.879\times10^{-1}$ & $3.915\times10^{-1}$ & $7.912\times10^{-1}$ & $7.949\times10^{-1}$ & $1.103\times10^{0}$ & $1.107\times10^{0}$ \\
$\mathrm{Error}_2$             & $3.741\times10^{-2}$ & $4.162\times10^{-2}$ & $2.092\times10^{-1}$ & $2.113\times10^{-1}$ & $1.408\times10^{-1}$ & $1.422\times10^{-1}$ \\
Av.\ Execution Time (s)        & $27.559$             & $24.273$             & $22.854$             & $22.378$             & $29.251$             & $27.849$             \\
Av.\ Memory Usage (MB)         & $162.992$            & $162.570$            & $182.121$            & $185.500$            & $183.512$            & $183.043$            \\
\bottomrule
\end{tabular}}
\end{table}

We now analyze the limit $h\to 0$ for IE. Table \ref{tab:errors_1} presents various values of $h$ alongside the corresponding $\text{Error}_1$ and $\text{Error}_2$ defined in \eqref{eq:definition_error}. Clearly, the error decreases as $h$ diminishes. Furthermore, \Cref{fig:conv} compares these errors with the linear rate of convergence predicted by \Cref{th:convergence} (and \Cref{coro:convergence}).

\begin{figure}[htbp]
  \centering
  \begin{subfigure}[t]{0.41\textwidth}
    \centering
    \includegraphics[width=\linewidth,valign=t]{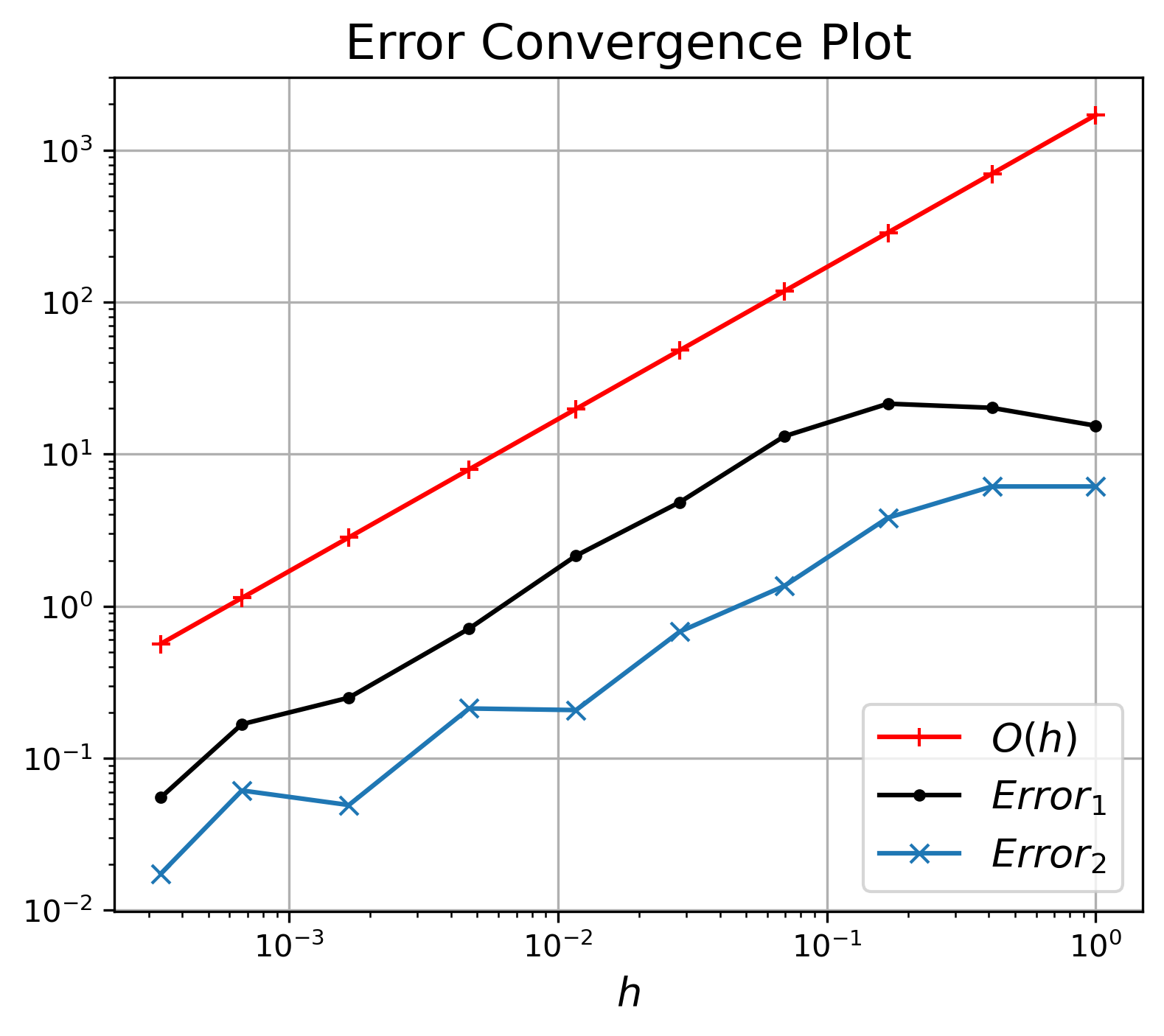}
    \caption{}\label{fig:conv}
  \end{subfigure}\hfill
  \begin{subtable}[t]{0.47\textwidth}
    \centering
    \adjustbox{valign=t}{%
     \resizebox{\textwidth}{!}{ \begin{tabular}{c|c|c}
        \toprule
        $h$                  & $\mathrm{Error}_1$    & $\mathrm{Error}_2$    \\
        \midrule
        $3.333\times10^{-4}$ & $5.530\times10^{-2}$ & $1.744\times10^{-2}$ \\
        $6.667\times10^{-4}$ & $1.675\times10^{-1}$ & $6.139\times10^{-2}$ \\
        $1.667\times10^{-3}$ & $2.502\times10^{-1}$ & $4.924\times10^{-2}$ \\
        $4.667\times10^{-3}$ & $7.118\times10^{-1}$ & $2.125\times10^{-1}$ \\
        $1.167\times10^{-2}$ & $2.150\times10^{0}$  & $2.074\times10^{-1}$ \\
        $2.833\times10^{-2}$ & $4.826\times10^{0}$  & $6.793\times10^{-1}$ \\
        $6.933\times10^{-2}$ & $1.310\times10^{1}$  & $1.358\times10^{0}$ \\
        $1.687\times10^{-1}$ & $2.147\times10^{1}$  & $3.818\times10^{0}$ \\
        $4.107\times10^{-1}$ & $2.016\times10^{1}$  & $6.128\times10^{0}$ \\
        $1.000\times10^{0}$   & $1.538\times10^{1}$  & $6.128\times10^{0}$ \\
        \bottomrule
      \end{tabular}}%
    }
    \vspace{3ex}
    \caption{}\label{tab:errors_1}
  \end{subtable}
  \vspace{-3ex}
  \caption{(a) Log–log plot of the errors versus the RBM step size $h$. The black curve with dots represents $\text{Error}_1$, the blue curve with crosses represents $\text{Error}_2$, and the red line with plus markers indicates the expected $\mathcal{O}(h)$ convergence. (b) lists the corresponding error values for different $h$.}
\label{fig:error-convergence}
\end{figure}


\subsection{Example 2}
Now, for the graph decomposition described in \Cref{fig_1_deamon_graph}, we consider the batches from {\bf option 1} introduced in \Cref{sec:example_decomp}; that is,
\begin{align}\label{eq:decompostion2}
    \mathcal{B}_1 =\{1\},\,\,\, \mathcal{B}_2 =\{2\},\,\, \,\mathcal{B}_3 =\{3\},\,\,\, \mathcal{B}_4 =\{4\},\,\,\, \mathcal{B}_5 =\{1,2,3\},\,\,\, \mathcal{B}_6 =\{2,3,4\}.
\end{align}
By setting the probabilities $p_j=1/6$ for $j\in\{1,\dots,6\}$, it follows that $\pi_1=\pi_4=1/3$ and $\pi_2=\pi_3=1/2$. With this configuration, \eqref{eq:semi_discrete_random} is solved using $100$ basis functions per edge, employing the previous time-schemes with a time step $\Delta t =0.002$, and for various values of $h\in\{0.002,0.004,0.006\}$. The corresponding results are presented in \Cref{table:rbm_dec2,table:rbm_dec2_theta_siem}. Note that the accuracy obtained with the batches \eqref{eq:decompostion2} is similar to that in \Cref{section:example1}. However, in this instance, the equation is never solved on the full graph, resulting in lower average memory usage compared to that in \Cref{table:rbm_dec1,table:rbm_dec1_theta_siem}. Moreover, RBM reduces memory consumption by approximately one-half relative to the full graph approach (see \Cref{tab:deterministic}).
\begin{table}[htbp]
\centering
\caption{Performance RBM with IE and CN, using decomposition \eqref{eq:decompostion2}.}\label{table:rbm_dec2}
\resizebox{\textwidth}{!}{
\begin{tabular}{l|cc|cc|cc}
\toprule
\multirow{2}{*}{\textbf{Metric}} 
    & \multicolumn{2}{c|}{$h=0.002$} 
    & \multicolumn{2}{c|}{$h=0.004$} 
    & \multicolumn{2}{c}{$h=0.006$} \\
\cmidrule(r){2-3} \cmidrule(r){4-5} \cmidrule(r){6-7}
    & \textbf{IE} & \textbf{CN} & \textbf{IE} & \textbf{CN} & \textbf{IE} & \textbf{CN} \\
\midrule
$\mathrm{Error}_1$             & $3.563\times10^{-1}$ & $3.503\times10^{-1}$ & $8.981\times10^{-1}$ & $9.134\times10^{-1}$ & $1.307\times10^{0}$ & $1.309\times10^{0}$ \\
$\mathrm{Error}_2$             & $7.238\times10^{-2}$ & $4.725\times10^{-2}$ & $1.796\times10^{-1}$ & $1.892\times10^{-1}$ & $2.656\times10^{-1}$ & $2.462\times10^{-1}$ \\
Av.\ Execution Time (s)        & $23.142$             & $33.798$             & $23.672$             & $35.799$             & $23.068$             & $34.339$             \\
Av.\ Memory Usage (MB)         & $106.582$            & $107.004$            & $106.672$            & $102.113$            & $101.703$            & $100.848$            \\
\bottomrule
\end{tabular}}
\end{table}

\begin{table}[htbp]
\centering
\caption{Performance RBM with $\theta$-method and SIEM, using decomposition \eqref{eq:decompostion2}.}\label{table:rbm_dec2_theta_siem}
\resizebox{\textwidth}{!}{
\begin{tabular}{l|cc|cc|cc}
\toprule
\multirow{2}{*}{\textbf{Metric}} 
    & \multicolumn{2}{c|}{$h=0.002$} 
    & \multicolumn{2}{c|}{$h=0.004$} 
    & \multicolumn{2}{c}{$h=0.006$} \\
\cmidrule(r){2-3} \cmidrule(r){4-5} \cmidrule(r){6-7}
    & $\boldsymbol\theta$ & \textbf{SIEM} & $\boldsymbol\theta$ & \textbf{SIEM} & $\boldsymbol\theta$ & \textbf{SIEM} \\
\midrule
$\mathrm{Error}_1$             & $3.551\times10^{-1}$ & $3.589\times10^{-1}$ & $8.952\times10^{-1}$ & $8.997\times10^{-1}$ & $1.305\times10^{0}$ & $1.309\times10^{0}$ \\
$\mathrm{Error}_2$             & $5.931\times10^{-2}$ & $7.029\times10^{-2}$ & $1.739\times10^{-1}$ & $1.774\times10^{-1}$ & $2.606\times10^{-1}$ & $2.632\times10^{-1}$ \\
Av.\ Execution Time (s)        & $23.952$             & $23.742$             & $23.941$             & $23.864$             & $23.741$             & $22.442$             \\
Av.\ Memory Usage (MB)         & $106.352$            & $107.762$            & $102.180$            & $102.109$            & $102.113$            & $101.734$            \\
\bottomrule
\end{tabular}}
\end{table}

\noindent
\textbf{Computational Setup.} The previous numerical experiments were executed on a device featuring an AMD Ryzen 9 5900HS processor with Radeon Graphics, clocked at $3.30$ GHz, and equipped with $16.0$ GB of RAM (with $15.4$ GB usable). The experiments were implemented in Python, utilizing the NumPy and SciPy libraries.
\subsection{On sampling variability and empirical averaging}\label{subsec:sample_analysis}

The computational times reported above correspond to average execution times. If the goal is to approximate the expected RBM trajectory, then independent realizations can be averaged. More precisely, let $z_h^{(1)},\ldots,z_h^{(\mathfrak m)}$ be independent copies of the RBM solution associated with the same RBM time window $h$, and set $\overline z_h^{\mathfrak m}(t):=\mathfrak m^{-1}\sum_{r=1}^{\mathfrak m} z_h^{(r)}(t)$. For each fixed $t\in[0,T]$, independence gives
\[
\mathbb E\|\overline z_h^{\mathfrak m}(t)-y(t)\|_{L^2(\mathcal E)}^2
=
\|\mathbb E[z_h(t)]-y(t)\|_{L^2(\mathcal E)}^2
+
\frac1{\mathfrak m}
\mathbb E\|z_h(t)-\mathbb E[z_h(t)]\|_{L^2(\mathcal E)}^2.
\]
The first term is the error of the expected RBM trajectory; it is controlled by the RBM time window $h$ and is not reduced by averaging more realizations. The second term is the variance of the empirical average in $L^2(\mathcal E)$; it decreases as $1/\mathfrak m$.

We illustrate this effect for the benchmark of \Cref{section:example1}. We compute $\mathfrak R$ independent realizations and, for each $\mathfrak m\leq\mathfrak R$, form empirical averages over subsets of $\mathfrak m$ realizations. With $\overline z_h^{\mathfrak R}:=\mathfrak R^{-1}\sum_{r=1}^{\mathfrak R} z_h^{(r)}$, we estimate the error of the expected RBM trajectory and the variance of the empirical average by
\[
\widehat{\mathfrak b}_h(t_n):=\|\overline z_h^{\mathfrak R}(t_n)-y(t_n)\|_{L^2(\mathcal E)}^2,\qquad
\widehat{\mathfrak f}_h(t_n):=\frac{1}{\mathfrak m}\left(\frac1{\mathfrak R}\sum_{r=1}^{\mathfrak R}
\|z_h^{(r)}(t_n)-\overline z_h^{\mathfrak R}(t_n)\|_{L^2(\mathcal E)}^2\right).
\]

\Cref{fig:sampling_variability} summarizes the effect of empirical averaging. Panel (a) shows the empirical error of $\overline z_h^{\mathfrak m}$, averaged over subsets of $\mathfrak m$ RBM realizations, together with the interquartile range. The error decreases rapidly for small values of $\mathfrak m$ and then approaches the estimated plateau. Note that in this example already a moderate number of realizations gives an error close to this plateau, indicating that the remaining error is mainly the error of the expected RBM trajectory rather than the finite-sample variance. Panel (b) displays the same mechanism through the estimated decomposition. The term $\widehat{\mathfrak f}_h$ decreases with $\mathfrak m$, whereas $\widehat{\mathfrak b}_h$ remains fixed. Consequently, $\widehat{\mathfrak b}_h+\widehat{\mathfrak f}_h$ approaches $\widehat{\mathfrak b}_h$ as $\mathfrak m$ increases. Thus, once the curve is close to the plateau, improving the accuracy requires reducing $h$, rather than increasing the number of averaged realizations. Finally, panel (c) compares computational times. Let $\tau_{\rm RBM}$ denote the average execution time of one RBM realization and $\tau_{\rm full}$ the execution time of the full-graph solver. If the $\mathfrak m$ realizations are computed sequentially, the corresponding time is approximately $\mathfrak m\tau_{\rm RBM}$. Since the realizations are independent, they can also be computed in parallel; with $\mathfrak p$ available processors, the ideal parallel time is approximately $\lceil\mathfrak m/\mathfrak p\rceil\tau_{\rm RBM}$, up to the final averaging overhead. Hence, empirical averaging increases the sequential time, but not necessarily the wall-clock time in a parallel implementation.

\begin{figure}[t]
    \centering
    \begin{subfigure}[t]{0.3\textwidth}
        \centering
        \includegraphics[width=\textwidth]{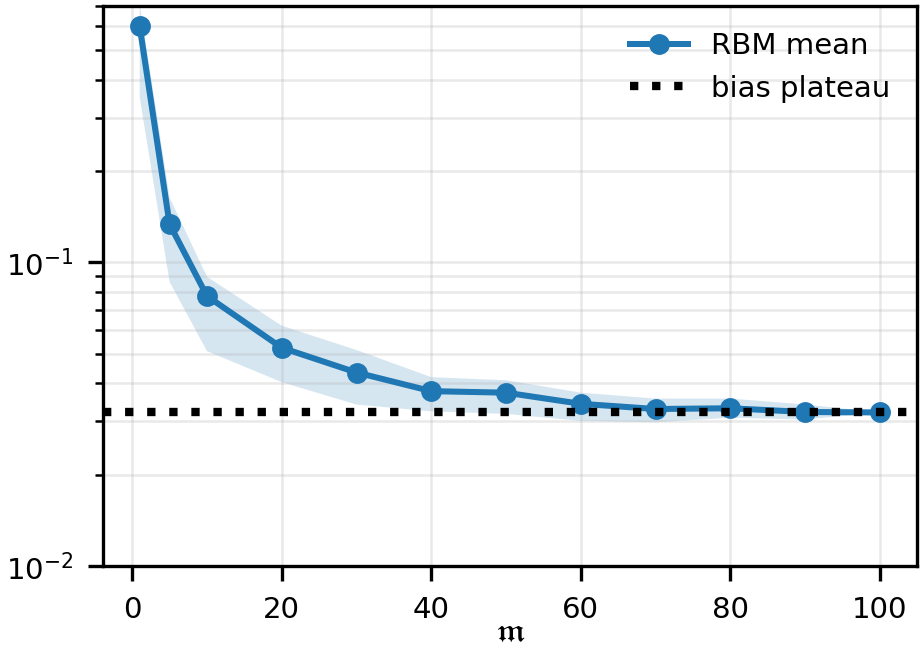}
        \subcaption{Empirical error}
    \end{subfigure}
    \begin{subfigure}[t]{0.3\textwidth}
        \centering
        \includegraphics[width=\textwidth]{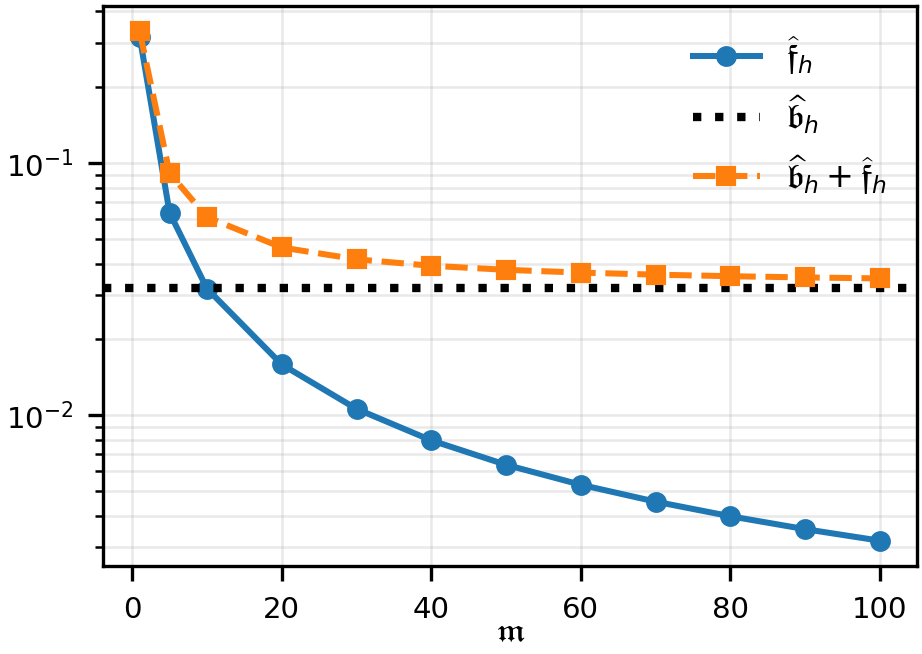}
        \subcaption{Error decomposition}
    \end{subfigure}
    \begin{subfigure}[t]{0.3\textwidth}
        \centering
        \includegraphics[width=\textwidth]{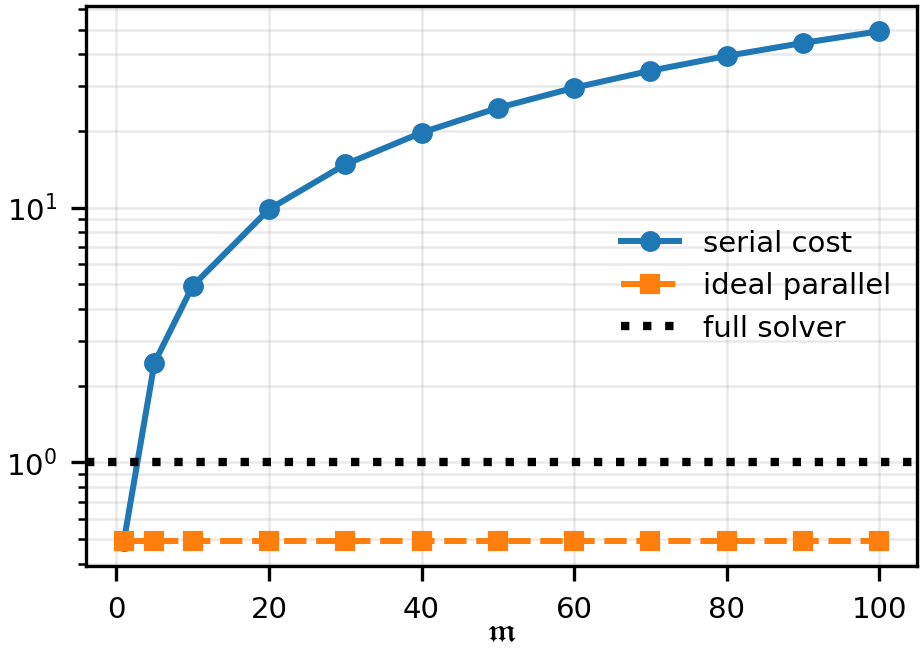}
        \subcaption{Computational time}
    \end{subfigure}
    \vspace{-0.4cm}
\caption{
Sampling variability and empirical averaging for the benchmark of \Cref{section:example1}. Panel (a) shows the empirical squared $L^2(\mathcal E)$ error of the average $\overline z_h^{\mathfrak m}$ over $\mathfrak m$ independent RBM realizations; the dotted line represents the estimated error of the expected RBM trajectory. Panel (b) shows the two terms in the error decomposition above, estimated from $\mathfrak R$ realizations, together with their sum. Panel (c) compares the sequential computational time with the ideal parallel time, both normalized by the execution time of the full graph solver.
}\label{fig:sampling_variability}
\end{figure}

\section{Conclusion and open problems}
In this work, we have introduced a fully continuous Random Batch Method for linear parabolic equations on graphs and proved in \Cref{th:convergence} that the mean-square error between the exact solution and the RBM approximation decays like $O(h)$. Numerical experiments confirm this convergence rate and show reductions in computational time and memory usage per RBM realization. These gains persist across different time discretization schemes. We also examined how empirical averaging over independent RBM realizations reduces sampling variability and how this affects the effective computational time.

We conclude by highlighting several directions for future work:
\smallbreak

\noindent
$\bullet$  \textbf{Optimal control.} Extending RBM to optimal control entails introducing a randomized version of the usual cost functional, penalizing both state and control in the RBM‐driven dynamics. In the distributed‐control setting with an $\alpha$‐convex (LQR) cost, one can adapt the energy‐estimate proofs from \cite{MR4433122}. For general Lipschitz (nonconvex) functionals, the Pontryagin principle and associated adjoint states must be invoked, and one might only expect convergence up to a constant reflecting the fact that small perturbations of the dynamics may drive to different local minimizers of the cost. Boundary‐control problems pose additional difficulties, since Theorem \ref{th:convergence} requires $H^1$‐in‐time boundary data; here, transposition methods or refined trace‐space frameworks may be necessary.


\noindent
$\bullet$ \textbf{Abstract operator framework.} One can ask whether RBM admits a formulation purely in operator‐theoretic terms: let $(A,D(A))$ be a densely defined operator on a Hilbert space $H$ generating an analytic semigroup, and assume initial data in the real interpolation space $(D(A),H)_{1/2}$ to secure maximal regularity.  In principle, the continuous RBM can then be recast by randomizing the action of $A$ (or its resolvent) on suitable subspaces, mirroring our edge‐wise decomposition.  However, proving convergence in this abstract setting will likely hinge on delicate issues such as the boundedness of the fractional power $A^{1/2}$ and the admissibility of any boundary (or control) operators in $H$, which may fail without additional structure.

\noindent
$\bullet$ \textbf{Sampling variability.}
The numerical study in \Cref{subsec:sample_analysis} quantifies the sampling effect for the benchmark considered in this work. A systematic analysis of how the variation of RBM trajectories depends on the data, graph topology, coefficients, and batch distribution remains open. This question is closely related to the sharpness of the variance estimate and to the effective complexity required to approximate the expected RBM trajectory within a prescribed tolerance. We refer the reader to \cite{alvarezlopez2025} for a related analysis in the context of nonlinear ODEs.

\noindent 
$\bullet$ \textbf{Optimal decomposition.}  
As the graph grows in size and complexity, it becomes essential to choose both the subgraph partition and the batch‐selection probabilities in a way that reflects the underlying PDE.  One may borrow from spectral graph theory—using eigenvectors of the weighted Laplacian to suggest cuts—but standard spectral weights ignore PDE data (coefficients, sources, initial conditions).  To make these partitions truly “physics‐informed,’’ one can incorporate problem data into the Laplacian weights before computing the spectrum.  An alternative is to train a graph-neural network on the dual graph, with problem data as input features, to predict which edges to retain in each batch.  Finally, another possibility is that any candidate decomposition can be evaluated by its induced variance operator $\Lambda$, and therefore the goal is to minimize $\|\Lambda\|_{L^1(0,T)}$. This approach minimizes the multiplicative error in Theorem~\ref{th:convergence} and speeds up convergence.  It remains open which of these strategies—spectral weighting, learned pruning, or variance‐based optimization—offers the best trade-off between computational cost and accuracy.

\appendix

\section{Functional spaces on graphs}\label{appendix1}

Let $\mathcal{G} := (\mathcal{E}, \mathcal{V}, l)$ be a metric graph with vertices $v \in \mathcal{V}$ and edges $e \in \mathcal{E}$, each having length $l_e \in (0,\infty)$. Define $\mathbb{N}_0 := \mathbb{N} \cup \{0\}$. Every edge $e \in \mathcal{E}$ is identified with the interval $[0,l_e]$, so that a function $y:\mathcal{G}\to\mathbb{R}$ corresponds to a collection $\{y_e\}_{e\in\mathcal{E}}$, where each $y_e:[0,l_e]\to\mathbb{R}$.

\subsection{Sobolev spaces}

We summarize the function spaces on graphs used in this work; see \cite[Chapter 3]{Mugnolo} for additional details.
\smallbreak

\noindent
\textbf{Piecewise spaces.} For $m \in \mathbb{N}_0$ and $p \in [1,\infty]$, the piecewise Sobolev space is
\[
W_{pw}^{m,p}(\mathcal{E})
:=
\Bigl\{\,y:\mathcal{E}\to\mathbb{R} \,\Big|\,
y_e \in W^{m,p}(0,l_e),\text{for each }e\in\mathcal{E}\Bigr\}.
\]
We write $L^p(\mathcal{E}):=W_{pw}^{0,p}(\mathcal{E})$ and $H_{pw}^m(\mathcal{E}) := W_{pw}^{m,2}(\mathcal{E})$. We endow $W_{pw}^{m,p}$ with the norm
\begin{align}
\|y\|_{W_{pw}^{m,p}(\mathcal{E})}^p
:=
\sum_{e\in\mathcal{E}}
\bigl\|y_e\bigr\|_{W^{m,p}(0,l_e)}^p,
\quad (1 \le p < \infty).
\end{align}

 For $k \in \mathbb{N}_0\cup\{\infty\}$, define the piecewise continuous space
\[
C_{pw}^k(\mathcal{E})
:=
\Bigl\{
  y:\mathcal{E}\to\mathbb{R}
  \,\Big|\,
  y_e \in C^k([0,l_e]),\text{for each }e\in\mathcal{E}
\Bigr\}.
\]
\textbf{Vertex-continuous spaces.} By \cite[Chap. 5.6, Thm. 6]{evans_book}, we have that $W^{m,p}(0,l_e)\subset C^{m-1,1-1/p}(0,l_e)$ if $m,\,p>1$. Thus, $W_{pw}^{m,p}(\mathcal{E}) \subset C^{m-1}_{pw}(\mathcal{E})$ if $m,\,p>1$. Thus, functions in $W_{pw}^{m,p}(\mathcal{E})$ are continuous on each edge $e$, but may be discontinuous at vertices. On the other hand, every function in $W^{m,p}(0,l_e)$ with $m\geq1$ and $p\in[1,\infty)$ has a well-defined trace (\cite[Chapter 5.5, Theorem 1]{evans_book}). This last observation motivates the Sobolev spaces of functions continuous at the vertices:
\[
W^{m,p}(\mathcal{E})
:=
\Bigl\{
  y \in W_{pw}^{m,p}(\mathcal{E})
  \,\Big|\,
  y_{e_1}(v) = y_{e_2}(v), 
  \forall v\in\mathcal{V}_0,\forall\,e_1,e_2\in\mathcal{E}(v)
\Bigr\},
\]
where $\mathcal{E}(v)$ is the set of all edges incident at $v$. For $p=2$, we write $H^m(\mathcal{E}):=W^{m,2}(\mathcal{E})$. If the graph has boundary vertices $\mathcal{V}_b\subset\mathcal{V}$, then
\[
H_0^m(\mathcal{E})
:=
\Bigl\{
  y \in H^m(\mathcal{E})
  \,\Big|\,
  y_e(v)=0,
  \text{for every }v\in\mathcal{V}_b
  \text{ and each } e\in\mathcal{E}(v)
\Bigr\}.
\]
It is noteworthy that the spaces $H^m(\mathcal{E})$ and $H_0^m(\mathcal{E})$ are Hilbert spaces since they are closed subspaces of the Hilbert space $\prod_{e\in \mathcal{E}} H^m(e)$ with (finite) constraints defined by continuous linear maps. Finally, for $k\in \mathbb{N}_0\cup\{\infty\}$, we define the space of $k$-times continuously differentiable functions on the graph by 
\begin{align}\label{eq:continuous_functions}
C^k(\mathcal{E})
:=
\Bigl\{
   y \in C_{pw}^k(\mathcal{E})
   \,\Big|\,
   \partial^k y_{e_1}(v)
   =
   \partial^k y_{e_2}(v),\,
   \forall v\in\mathcal{V}_0,\,
   e_1,e_2 \in\mathcal{E}(v)
\Bigr\}.
\end{align}
Here $\partial^k y$ denotes the $k$-th derivative of $y$. If $k=0$, we write $C(\mathcal{E}):=C^0(\mathcal{E})$. For simplicity, we retain the notation $C(\mathcal{E})$, even though functions in $C(\mathcal{E})$ belong to $C(\mathcal{G})$ and admit well-defined values at the vertices. For $k=\infty$, we let $C^\infty(\mathcal{E})$ be the set of infinitely differentiable functions satisfying vertex continuity of all derivatives. Moreover,
$C^\infty_c(\mathcal{E})
\subset
C^\infty(\mathcal{E})$ denotes those functions with compact support in $\mathcal{E}$, and $C^\infty_0(\mathcal{E})
\subset
C^\infty(\mathcal{E})$ denotes those functions vanishing at every boundary vertex in $\mathcal{V}_b$.

\smallbreak

\noindent
\textbf{Spaces on vertices.} For $p\in[1,\infty)$, we consider $l^p(\mathcal{V})$, the space of real-valued functions defined on the set of vertices $\mathcal{V}$, given by
\begin{align*}
l^p(\mathcal{V})
=
\Bigl\{
  f:\mathcal{V} \to \mathbb{R}
  \Big|
  \|f\|_{l^p(\mathcal{V})}
  :=
  \Bigl(\sum_{v \in \mathcal{V}} |f(v)|^p\Bigr)^{1/p}
  < \infty
\Bigr\}.
\end{align*}
In the remainder, we focus primarily on the space $l^2(\mathcal{V}_b)$.
\smallbreak

\noindent
{\bf Bochner spaces.} Finally, we also consider the standard definitions of Bochner spaces $L^p(0,T;X)$, continuous spaces $C([0,T];X)$, and Sobolev spaces $W^{m,p}(0,T;X)$ (with $H^m(0,T;X)=W^{m,2}(0,T;X)$), where $X$ will be taken as a functional space defined on a graph. See \cite[Section 5.9.2]{evans_book} for a comprehensive introduction of Bochner spaces.

\subsection{Density results}
Regarding the relation between the previous spaces, we have the following lemmas.

    \begin{lemma}\label{lemma:dense_CinfH1}
The space $C^\infty_0(\mathcal{E})$ is dense in $H_0^1(\mathcal{E})$.
\end{lemma}

\begin{proof} Take any $u \in H_0^1(\mathcal{E})$. By definition: $u_e \in H^1(e)$ for each $e \in \mathcal{E}$, $u_e(v) = 0$ for every $v\in \mathcal{V}_b$, and $ u_{e_1}(v) = u_{e_2}(v)$ for every $v\in \mathcal{V}_0$, and $e_1,\,e_2 \in \mathcal{E}(v)$. We will approximate $u$ in $H_0^1(\mathcal{E})$ by functions $\rho$ belonging to $C^\infty_0(\mathcal{E})$. We proceed in two steps:
\smallbreak

\noindent
\textbf{Step 1. Approximation on each edge.}
Let $e \in \mathcal{E}$ be fixed. We have two cases:

\noindent
\emph{-Case (a). $e$ meets a boundary vertex $v \in \mathcal{V}_b$.}
In this case  $u_e \in H^1_v(e)$ where $H^1_v(e)$ denotes the subspace of $H^1(e)$ whose functions also vanish at $v$. Then it follows from \cite[Section 8.3]{brezis_book} that $C^\infty_{0(v)}(\overline{e})$ is dense in $H^1_v(e)$. This implies the existence of a sequence $\{\rho_e^n\}_{n\in \N}\subset C^\infty_{0(v)}(\overline{e})$ with $\rho_e^n \to u_e$ in $H^1(e)$. Here $C^\infty_{0(v)}(\overline{e})$ denotes smooth functions on $\overline{e}$ that vanish at $v$.

\noindent
\emph{-Case (b). $e$ does not meet any boundary vertex.}
Then $u_e \in H^1(e)$, and by \cite[Section 5.3.3]{evans_book}, $C^\infty(\overline{e})$ is dense in $H^1(e)$. Hence there is a sequence $\{\rho_e^n\}_{n\in \N}\subset C^\infty(\overline{e})$ such that $\rho_e^n\to u_e$ in $H^1(e)$.
\smallbreak

\noindent
\textbf{Step 2. Global smooth function.}
Let $v \in \mathcal{V}_0$ be an interior vertex, and suppose it is incident to $|E| \ge 2$ edges, labeled $e_1,\dots,e_k \in \mathcal{E}(v)$. Since $u_{e_i}(v)$ is the same for all $i\in[|E|]$, each sequence $\{\rho_{e_i}^n\}_{n\in \N}$ converges to $u_{e_i}$ in $H^1(e_i)$, and by Sobolev embedding $H^1(e_i)\hookrightarrow C(\overline{e_i})$, we have $\rho_{e_i}^n(v)\to u(v)$ in the supremum norm for each $i$. Hence for large $n$, the differences $|\rho_{e_i}^n(v)-\rho_{e_j}^n(v)|$ become arbitrarily small, but may not vanish. To obtain a family $\{\tilde{\rho}_{e_i}^n\}_{n\in\N}$, whose values \emph{exactly} coincide at $v$, for each $j\in \{2,\dots,|E|\}$, and every $n\in \N$ define
\begin{equation*}
    \tilde{\rho}_{e_j}^n(x)
:=
\rho_{e_j}^n(x) 
+ 
\eta_v^n(x)\,\bigl(\rho_{e_1}^n(v)-\rho_{e_j}^n(v)\bigr),
\quad
\tilde{\rho}_{e_1}^n(x) := \rho_{e_1}^n(x),
\end{equation*}
where $\eta_v^n$ is a smooth cutoff supported in a small neighborhood of $v$ (on edge $e_j$) such that $\eta_v^n(v)=1$ and $\|\eta_v^n\|_{H^1(e_j)}$ can be made arbitrarily small. This adjustment forces $\tilde{\rho}_{e_1}^n(v) 
=
\tilde{\rho}_{e_2}^n(v)
=\cdots=
\tilde{\rho}_{e_{|E|}}^n(v),$ for every $n\in \N$. Finally, by defining $\psi^n \colon \mathcal{E}\to \mathbb{R}$ as $\psi^n|_e =\tilde{\rho}_e^n$, we have that $\{\psi^n\}_{n\in \N}\subset C^\infty_0(\mathcal{E})$. Moreover, by construction $\psi^n\to u$ in $H_0^1(\mathcal{E})$ as $n\to \infty$. Concluding the proof.
\end{proof}

\begin{lemma}\label{lemma:dense_H1L2}
The space $H_0^1(\mathcal{E})$ is dense in $L^2(\mathcal{E})$.
\end{lemma}

\begin{proof}
Since $C_0^\infty(e)$ is dense in $L^2(e)$ for each $e \in \mathcal{E}$, it follows that $C_0^\infty(\mathcal{E})$ is dense in $L^2(\mathcal{E})$. Because $C_0^\infty(\mathcal{E}) \subset H_0^1(\mathcal{E}) \subset L^2(\mathcal{E})$, we can take closure in the $L^2(\mathcal{E})$ topology, and deduce the result.
\end{proof}
\begin{lemma}\label{lemma:dense_L2H-1}
Let $H^{-1}(\mathcal{E})$ be the dual of $H_0^1(\mathcal{E})$. Then $L^2(\mathcal{E})$ is dense in $H^{-1}(\mathcal{E})$.
\end{lemma}

\begin{proof}
For every $f\in L^2(\mathcal{E})$ consider $f \mapsto F_f,$ where $F_f(u) = (f, u)_{L^2(\mathcal{E})}$ for all $u \in H_0^1(\mathcal{E}).$ Assume for contradiction that $ L^2(\mathcal{E}) $ is not dense in $ H^{-1}(\mathcal{E}) $. By the Hahn-Banach theorem, there exists a non-zero linear functional $ \Phi \in (H^{-1}(\mathcal{E}))^*$ such that $\Phi(F_f) = 0,$ for all $f \in L^2(\mathcal{E}).$ By reflexivity, $(H^{-1}(\mathcal{E}))^*$ is isometrically isomorphic to $ H_0^1(\mathcal{E})$; hence there exists a non-zero $u\in H_0^1(\mathcal{E})$ satisfying $\Phi(F) = F(u)$ for all $F \in H^{-1}(\mathcal{E}).$ In particular, for all $ f \in L^2(\mathcal{E}) $ we have $0 = \Phi(F_f) = F_f(u) = (f, u)_{L^2(\mathcal{E})}.$ This implies $u =0 $ in $ L^2(\mathcal{E})$ and consequently $u =0 $ in $ H_0^1(\mathcal{E})$, contradicting $u \neq 0$. Hence, $L^2(\mathcal{E})$ must be dense in $H^{-1}(\mathcal{E}) $.
\end{proof}

 \section{Properties of the main operator on the graph}\label{appendix3}

\subsection{M-accretivity  and analyticity}
\begin{lemma}\label{prop:A_deg_generator}
Consider
$A : D(A) \subset L^2(\mathcal{E}) \to L^2(\mathcal{E})$ defined in \eqref{def_A_heat}. Then $(A+\lambda I)$ is $M$-accretive for $\lambda>0$, and $A$ generates an analytic semigroup 
on $L^2(\mathcal{E})$.
\end{lemma}

\begin{proof}
We will start first by showing the M-accretivity of the operator $A$.
\smallbreak

\noindent
\textbf{1. M-accretivity.}
We split the argument into two parts: 

\smallbreak
\textit{(i) Proof of accretivity.} Let $y \in D(A)$. By integrating by parts,
\begin{align*}
(Ay,y)_{L^2(\mathcal{E})} 
&= 
\bigl(-\partial_x\bigl(a\partial_x y\bigr),y\bigr)_{L^2(\mathcal{E})}
+\bigl(b\partial_x y,y\bigr)_{L^2(\mathcal{E})}
+\bigl(py,y\bigr)_{L^2(\mathcal{E})}
\\
&= 
\bigl(a\,\partial_x y,\partial_x y\bigr)_{L^2(\mathcal{E})}
-
\sum_{v \in \mathcal{V}}\sum_{e \in \mathcal{E}(v)}
a_e(v)y_e(v)\,\partial_x y_e(v)n_e(v)\\
&\qquad+\bigl(b\partial_x y,y\bigr)_{L^2(\mathcal{E})}+\bigl(py,y\bigr)_{L^2(\mathcal{E})}.
\end{align*}
Because $y(v)=0$ for every $v \in \mathcal{V}_b$ and 
\begin{align*}
\sum_{e \in \mathcal{E}(v)} a_e(v)\,y_e(v)\,\partial_x y_e(v)\,n_e(v) 
= 0
\quad 
\text{for all } v \in \mathcal{V}_0,
\end{align*}
the boundary term vanishes, giving
\begin{align*}
(Ay,y)_{L^2(\mathcal{E})}
=
\bigl(a\,\partial_x y,\,\partial_x y\bigr)_{L^2(\mathcal{E})}
+\bigl(b\,\partial_x y,\,y\bigr)_{L^2(\mathcal{E})}
+\bigl(p\,y,\,y\bigr)_{L^2(\mathcal{E})}.
\end{align*}
Using Hölder’s and Young’s inequalities and the uniform ellipticity of $a$, we deduce
\begin{align*}
(Ay,y)_{L^2(\mathcal{E})}
&\geq 
a_0\|\partial_x y\|_{L^2(\mathcal{E})}^2
-
\frac{\varepsilon\|b\|_{L^\infty(\mathcal{E})}^2}{2}\|\partial_x y\|_{L^2(\mathcal{E})}^2-\frac{\|b\|_{L^\infty(\mathcal{E})}^2}{2\varepsilon}\|y\|_{L^2(\mathcal{E})}^2\\
&\qquad
-
\|p\|_{L^\infty(\mathcal{E})}\,\|y\|_{L^2(\mathcal{E})}^2\\
&\geq \left(a_0-\frac{\varepsilon\|b\|_{L^\infty(\mathcal{E})}^2}{2}\right)\|\partial_x y\|_{L^2(\mathcal{E})}^2 - \left(\frac{\|b\|_{L^\infty(\mathcal{E})}^2}{2\varepsilon}+\|p\|_{L^\infty(\mathcal{E})}\right)\|y\|_{L^2(\mathcal{E})}^2\\
&\geq C\|\partial_x y\|_{L^2(\mathcal{E})}^2 -\lambda\|y\|_{L^2(\mathcal{E})}^2, 
\end{align*}
Therefore, by choosing $\varepsilon = a_0/\|b\|_{L^\infty(\mathcal{E})}^2$, and using Poincare's inequality, we get
\begin{align}\label{eq:momotonicity}
    ((A+\lambda I )y,y)_{L^2(\mathcal{E})}\geq \frac{a_0C_p}{2}\| y\|_{H_0^1(\mathcal{E})}^2\geq 0,
\end{align}
with $\lambda = \|b\|_{L^\infty(\mathcal{E})}^2/2\varepsilon+\|p\|_{L^\infty(\mathcal{E})}$ and $C_p>0$ being the Poincaré constant. Hence, $A +\lambda I$ is accretive.

\smallbreak
\textit{(ii) Surjectivity of  $A+\lambda I$.} To conclude the $M$-accretivity, it is enough to show that
$\text{Range}(\lambda I + A) = L^2(\mathcal{E})$. Let us consider the bilinear form $\mathfrak{a}:H_0^1(\mathcal{E})\times H_0^1(\mathcal{E})\to \R$
    \begin{align*}
        \mathfrak{a}_\lambda(y,z)= (a\partial_x y, \partial_x z)_{L^2(\mathcal{E})} +(b\partial_x y,z)_{L^2(\mathcal{E})} + (py,z)_{L^2(\mathcal{E})}+(\lambda y,z)_{L^2(\mathcal{E})}.
    \end{align*}
To verify that $\mathfrak{a}$ is continuous, let us apply Hölder’s and Young’s inequalities to get
\begin{align}\label{eq:dissip2}
   \mathfrak{a}_\lambda (y,y)&\leq\left( \|a\|_{L^\infty(\mathcal{E})} +\|b\|_{L^\infty(\mathcal{E})} +\|p\|_{L^\infty(\mathcal{E})} +\lambda  \right) \| y \|_{H^1_0(\mathcal{E})}^2.
\end{align}
Moreover, analogous to \eqref{eq:momotonicity}, the coercivity of $\mathfrak{a}_\lambda$ follows. Therefore, Lax--Milgram theorem guarantees for every $f\in L^2(\mathcal{E}) $ the existence a unique 
$y\in H_0^1(\mathcal{E})$ such that
\begin{align*}
\mathfrak{a}_\lambda(y,z)
=
(f,z)_{L^2(\mathcal{E})},
\quad \text{for every }
z\in H_0^1(\mathcal{E}).
\end{align*}
In particular, we deduce that  
\begin{align}\label{eq:variation_z_v_Cinf}
    (a\partial_x y,\partial_x z)_{L^2(\mathcal{E})}= -( b\partial_x y +(p+\lambda)y -f,z)_{L^2(\mathcal{E})} \quad \text{for every }z\in C^\infty_{c}(\mathcal{E}).
\end{align}
Therefore, the distributional derivative of $a\partial_x y$ is given by 
\begin{align}\label{eq:distriv_deriv_y}
    \partial_x(a\partial_x y) = b\partial_x y +(p+\lambda)y -f\in L^2(\mathcal{E}).
\end{align}
Since $a\in W^{1,\infty}_{pw}(\mathcal{E})$, we deduce that $y\in H^2(\mathcal{E})\cap H_0^1(\mathcal{E})$. On the other hand, since $C^\infty_0(\mathcal{E})$ is densely embedded into $H^1_{0}(\mathcal{E})$, integrating by parts \eqref{eq:variation_z_v_Cinf} we get
\begin{align}\label{eq:des_var_Cinf}
    (-\partial_x(a\partial_x y) + b\partial_x y +(p+\lambda)y -f,z)_{L^2(\mathcal{E})}+\sum_{v\in \mathcal{V}}\sum_{e\in\mathcal{E}(v)}a_e(v)z_e(v)\partial_x y_e(v)n_e(v)=0,
\end{align}
for every $z\in H^1_0(\mathcal{E})$. Since $z(v)=0$ for every $v\in \mathcal{V}_b$, and $z_e(v)=z_v$ for every $v\in\mathcal{V}_0$, using \eqref{eq:distriv_deriv_y}, \eqref{eq:des_var_Cinf} is reduced to
\begin{align*}
     0&=(-\partial_x(a\partial_x y) + b\partial_x y +(p+\lambda)y -f,z)_{L^2(\mathcal{E})}+\sum_{v\in \mathcal{V}}\sum_{e\in\mathcal{E}(v)}a_e(v)z_e(v)\partial_x y_e(v)n_e(v)\\
     &=  \sum_{v\in \mathcal{V}_0}z_v\sum_{e\in\mathcal{E}(v)} a_e(v)\partial_x y_e(v)n_e(v),
\end{align*}
for all $z\in H^1_0(\mathcal{E})$. Consequently, $y\in D(A)$, and therefore, $A+\lambda I$ is M-accretive. In consequence of Lumer-Phillips theorem (see \cite[Theorem 2.8]{Bensoussan}), $A+\lambda I$ generates a semigroup of contractions on $L^2(\mathcal{E})$, implying that $A$ generates a semigroup in $L^2(\mathcal{E})$.

\smallskip
\noindent
\textbf{2. Analytic semigroup.}
Now let us prove that this semigroup is analytic. For that let us consider the bilinear form $\mathfrak{b}:H_0^1(\mathcal{E})\times H_0^1(\mathcal{E})\to \R$ defined by     \begin{align}\label{eq:bilinear_form_b}
    \mathfrak{b}(y,z)= (a\partial_x y,\partial_xz )_{L^2(\mathcal{E})} +(b\partial_x y,z)_{L^2(\mathcal{E})} + (py,z)_{L^2(\mathcal{E})}.
\end{align}
Observe that the bilinear form $\mathfrak{b}$ is continuous. Moreover, in a similar way to \eqref{eq:momotonicity}, the bilinear form $\mathfrak{b}$ satisfies the Gårding inequality (or $H-V$ coercivity), i.e.,
\begin{align}\label{eq:VHcoercitivity}
\mathfrak{b}(y,y) +\lambda\| y\|^2_{L^2(\mathcal{E})} \geq \beta\|y \|^2_{H_0^1(\mathcal{E})},\quad y\in H^1_0(\mathcal{E}),
\end{align}
for some $\lambda,\,\beta>0$ ($\lambda$ can be taken as in \eqref{eq:momotonicity} and $\beta =a_0/2$). Now, consider the operator $\mathfrak{\mathcal{B}}$ associated to $\mathfrak{b}$
 defined by 
 \begin{align*}
     \begin{cases}
         D(\mathfrak{\mathcal{B}}) = \{y\in H_0^1(\mathcal{E})\,;\, z\mapsto \mathfrak{b}(y,z) \text{ is } L^2(\mathcal{E}) -\text{continuous}\},\\
         (-\mathfrak{\mathcal{B}}y,z)=\mathfrak{b}(y,z),\quad \text{for every }y\in D(\mathfrak{\mathcal{B}}),\, z\in H_0^1(\mathcal{E}). 
     \end{cases}
 \end{align*}
We will show that, in fact, $(\mathfrak{\mathcal{B}},D(\mathfrak{\mathcal{B}})) =(A,D(A))$. Let us characterize $D(\mathfrak{\mathcal{B}})$, for that, integrating by parts the first product of $\mathfrak{b}$, we observe
\begin{align}\label{eq:first_integration_bilinear}
   \nonumber \mathfrak{b}(y,z) &= (-\partial_x (a\partial_x y),z)_{L^2(\mathcal{E})}+(b\partial_x y,z)_{L^2(\mathcal{E})} +(p y,z)_{L^2(\mathcal{E})}\\ &\quad+\sum_{v\in\mathcal{V}}\sum_{e\in\mathcal{E}(v)}a_e(v)z_e(v)\partial_x y_e(v)n_e(v). 
\end{align}
Then, using the fact that $z_e(v)=0$ for every $v\in\mathcal{V}_b$, and applying H\"older's inequality 
\begin{align*}
    \mathfrak{b}(y,z) &\leq \left(\|\partial_x (a\partial_x y)\|_{L^2(\mathcal{E})}+\|b\partial_x y\|_{L^2(\mathcal{E})} +\|p y\|_{L^2(\mathcal{E})}\right)\|z\|_{L^2(\mathcal{E})}\\ &\quad+\sum_{v\in\mathcal{V}_0}\sum_{e\in\mathcal{E}(v)}a_e(v)z_e(v)\partial_x y_e(v)n_e(v). 
\end{align*}
Therefore, for $z\mapsto \mathfrak{b}(y,z) $ to be continuous into $L^2(\mathcal{E})$, we require $y\in H^2_{pw}(\mathcal{E})$ and $\sum_{e\in\mathcal{E}(v)}a_e(v)\partial_x y_e(v)n_e(v)=0.$ Consequently, $y\in D(A)$ and we have that $D(\mathfrak{\mathcal{B}})\subset D(A)$, a completely analogous argument show that $D(A) \subset D(\mathfrak{\mathcal{B}})$. Thus, from \eqref{eq:first_integration_bilinear}, it follows that $\mathfrak{\mathcal{B}}=A$. Finally, due to \Cref{lemma:dense_H1L2} and  \Cref{lemma:dense_L2H-1},  $(H_0^1(\mathcal{E}),L^2(\mathcal{E}),H^{-1}(\mathcal{E}))$ defines a Gelfant triplet. It follows from \cite[Theorem 2.12]{Bensoussan} that, due to \eqref{eq:VHcoercitivity}, we can conclude that $A$ generates an analytic semigroup.
\end{proof}

\subsection{Maximal regularity}

\begin{lemma}\label{lemma:fractional_power}
Let $A:D(A)\subset L^2(\mathcal{E})\to L^2(\mathcal E)$ be the operator defined in \eqref{def_A_heat}. Assume $(a,b,p)\in W^{1,\infty}_{pw}(\mathcal{E})\times L^{\infty}(\mathcal{E})\times L^{\infty}(\mathcal{E})$ satisfy \eqref{eq:assumption_coef_a} and \eqref{eq:assumption_coef_b}. Consider 
\begin{align*}
    D_{A}(1/2,2):=\{u(0)\,:\, u\in W(2,0; D(A),L^2(\mathcal{E}))\},
\end{align*}
where $ W(2,0;D(A),L^2(\mathcal{E})):= H^1(0,T;L^2(\mathcal{E}))\cap L^2(0,T; D(A))$. Then, $D_{A}(1/2,2)$ coincides with the space $H^1_{0}(\mathcal{E})$. 
\end{lemma}
\begin{proof}
The argument is carried out in two steps.

\smallbreak
\noindent
\textbf{Step 1. Identification via interpolation.}
By \cite[Proposition~4.3]{Bensoussan}, one obtains $D_{A}(1/2,2) \simeq \bigl(D(A),\,L^2(\mathcal{E})\bigr)_{1/2},$ where the space on the right-hand side denotes the real interpolation space. To characterize this interpolation space, first note that $(A + \lambda I)$ is $M$-accretive by \Cref{prop:A_deg_generator}, and due to \eqref{eq:momotonicity}, $(A + \lambda I)$ is injective. It follows that $(A + \lambda I)^{-1}$ is well-defined.

We now verify that $(A + \lambda I)^{-1}$ is a bounded operator on $L^2(\mathcal{E})$. Let $f \in L^2(\mathcal{E})$ and define $y := (A+\lambda I)^{-1} f$. If $y=0$, the claim is immediate. Otherwise, applying Poincaré's inequality, Hölder's inequality, and \eqref{eq:momotonicity} yields,
\begin{align*}
\|y\|_{L^2(\mathcal{E})}^2 \,\frac{a_0}{2C_p}
\le
\bigl((A+\lambda I)\,y,y\bigr)_{L^2(\mathcal{E})}
=
\bigl(f,y\bigr)_{L^2(\mathcal{E})}
\le
\|f\|_{L^2(\mathcal{E})}\,\|y\|_{L^2(\mathcal{E})}, 
\end{align*}
for every $f\in L^2(\mathcal{E})$. Dividing by $\|y\|_{L^2(\mathcal{E})}$ shows that $(A + \lambda I)^{-1}$ is indeed bounded on $L^2(\mathcal{E})$. Hence, by \cite[Proposition~6.1]{Bensoussan}, we have that
\begin{align}\label{eq:interpolations_frac}
\bigl(D(A),\,L^2(\mathcal{E})\bigr)_{1/2}
=
\bigl(D(A+\lambda I),\,L^2(\mathcal{E})\bigr)_{1/2}
=
D\bigl((A+\lambda I)^{1/2}\bigr)
=
D\bigl(A^{1/2}\bigr).
\end{align}

\smallbreak
\noindent
\textbf{Step 2. Characterization of the domain of $A^{1/2}$.}
Let us characterize $A^*$, the adjoint operator of $A$. Consider $y \in D(A)$. Integrating by parts gives
\vspace{-0.1cm}
\begin{align*}
(Ay,z)
&=
-\bigl(\partial_x\bigl(a\,\partial_x y\bigr),z\bigr)_{L^2(\mathcal{E})}
+\bigl(b\partial_x y,z\bigr)_{L^2(\mathcal{E})}
+\bigl(py,z\bigr)_{L^2(\mathcal{E})}
\\
&=
-\bigl(y,\partial_x\bigl(a\partial_x z\bigr)\bigr)_{L^2(\mathcal{E})}
-
\bigl(y,\partial_x (bz)\bigr)_{L^2(\mathcal{E})}
+
\bigl(y,pz\bigr)_{L^2(\mathcal{E})}
\\
&\quad
\hspace{-1.37cm}+\!\sum_{v\in \mathcal V}
\sum_{e\in \mathcal{E}(v)}\!\!\!
\Bigl(
a_e(v)y_e(v)\partial_x z_e(v)n_e(v)
\!-\!
a_e(v)\partial_x y_e(v)z_e(v)n_e(v)
\!+\!
y_e(v)z_e(v)b_e(v)n_e(v)
\Bigr).
\end{align*}
For the products to be well-defined in $L^2(\mathcal{E})$, one needs $z\in H^2_{pw}(\mathcal{E})$. We impose vertex conditions for $z$ such that the boundary term vanishes. Since $y\in H_0^1(\mathcal{E})$, setting 
\vspace{-0.1cm}
\begin{align*}
\sum_{e\in \mathcal{E}(v)} a_e(v)\partial_x z_e(v)\,n_e(v) = 0,
\quad\text{for every }v\in V_0,
\end{align*}
\vspace{-0.3cm}

\noindent
the first boundary term vanishes. Further, because $y\in D(A)$, requiring $z_e\in H_0^1(\mathcal{E})$ makes the second boundary term vanishing as well. Lastly, the remaining boundary term vanishes since $z\in H_0^1(\mathcal{E})$ and \eqref{eq:assumption_coef_b} holds. Consequently, for all $y,z\in D(A)$,
\begin{align*}
(Ay,z)
=
\bigl(y,\,
- \partial_x\bigl(a\,\partial_x z\bigr)
- \partial_x (bz)
+ pz
\bigr)_{L^2(\mathcal{E})}
=
(y,A^*z),
\end{align*}
which shows $D(A^*) = D(A)$. Applying \cite[Theorem~8.2]{MR2124040} (see also \cite{Lions1962EspacesDE}) we deduce that $D\bigl(A^{1/2}\bigr)
=
D\bigl(A^{*1/2}\bigr)
=
D(\mathfrak{b})
=
H_0^1(\mathcal{E}),$ where $\mathfrak{b}$ is the bilinear form associated with $A$ from \eqref{eq:bilinear_form_b}. Finally, \eqref{eq:interpolations_frac} completes the argument.
\end{proof}

\begin{theorem}\label{th:maximal_regularity}
Let $A:D(A)\subset L^2(\mathcal{E})\to L^2(\mathcal{E})$ be the operator defined in \eqref{def_A_heat}. Then, if $u^0\in H^1_0(\mathcal{E})$ and $f\in L^2(0,T;L^2(\mathcal{E}))$, the equation 
\begin{align}\label{eq:cauchy_problem_homo}
    \begin{cases}
        \partial_t u + Au = f, \quad t\in (0,T),\\
    u(0)=u^0,
    \end{cases}
\end{align}
has a unique solution $u$ in $ H^1(0,T;L^2(\mathcal{E}))\cap L^2(0,T;D(A))\cap C([0, T]; H_0^1(\mathcal{E})).$ Moreover, $u$ satisfies the following inequality
\begin{align*}
        \|u\|_{H^1(0, T; L^2(\mathcal{E}))} + \|u\|_{L^2(0, T; D(A))} \leq C \left( \|f\|_{L^2(0, T; L^2(\mathcal{E}))} + \|u^0\|_{H^1_0(\mathcal{E})} \right).
    \end{align*} 
\end{theorem}
\begin{proof}
    Due to \Cref{prop:A_deg_generator}, $A$ generates an analytic semigroup. Then, we know from \cite[Theorem 3.1]{Bensoussan} that the map 
    \begin{align*}
    \Phi:H^1(0,T;L^2(\mathcal{E}))\cap L^2(0,T;D(A))& \to  L^2(0,T;L^2(\mathcal{E}))\times D_A(1/2,2)\\
    u&\mapsto (\partial_t u +Au,u(0))
    \end{align*}
    is an isomorphism. On the other hand, due to \cref{lemma:fractional_power}, we have that $D_{A}(1/2,2)$ coincides with the space $H^1_{0}(\mathcal{E})$. Therefore, under the assumptions of \Cref{th:maximal_regularity}, we can guarantee that  \eqref{eq:cauchy_problem_homo} has a unique solution $u$ in   $H^1(0,T;L^2(\mathcal{E}))\cap L^2(0,T;D(A))$. Finally, using the analyticity of $A$, \cite[Remark 4.2]{Bensoussan} ensures that $H^1(0,T;L^2(\mathcal{E}))\cap L^2(0,T;D(A))\subset C([0, T]; H_0^1(\mathcal{E})),$ concluding the proof.
\end{proof}

	\section*{Acknowledgments}
The author wishes to express sincere gratitude to H. Meinlschmidt for valuable discussions on interpolation spaces, and to S. Zamorano for a careful and critical reading of the manuscript. The author is also grateful to E. Zuazua for his unconditional support and encouragement.

\vspace{-0.3cm}

\bibliographystyle{abbrv}
\bibliography{biblio}

\end{document}